\documentclass{amsart}
\usepackage[latin1]{inputenc}
\usepackage[T1]{fontenc}
\usepackage{amsmath}
\usepackage{amssymb}
\usepackage{amsfonts}
\usepackage{amsthm}
\usepackage{graphicx}
\usepackage{color}
\usepackage[all]{xy}
\usepackage{sidecap}
\usepackage{times}
\DeclareMathOperator{\dist}{dist}

\DeclareMathOperator*{\esssup}{ess\,sup}

\DeclareMathOperator{\sgn}{sgn}
\DeclareMathOperator{\spann}{span}
\DeclareMathOperator{\id}{id}
\DeclareMathOperator{\ind}{ind}
\DeclareMathOperator{\coker}{coker}
\DeclareMathOperator{\codim}{codim}
\newcommand{\C}{{\mathbb C}}
\newcommand{\N}{{\mathbb N}}
\newcommand{\R}{{\mathbb R}}
\newcommand{\Z}{{\mathbb Z}}
\newcommand{\fB}{{\mathfrak B}}
\newcommand{\cB}{{\mathcal B}}
\newcommand{\cC}{{\mathcal C}}
\newcommand{\cS}{{\mathcal S}}
\newcommand{\cV}{{\mathcal V}}
\renewcommand{\d}{\,{\mathrm d}}
\newcommand{\tm}{\times}
\newcommand{\eps}{\varepsilon}
\newcommand{\intoo}[1]{\left(#1\right)}		
\newcommand{\intcc}[1]{\left[#1\right]}		
\newcommand{\set}[1]{\left\{#1\right\}}		
\newcommand{\abs}[1]{\left|#1\right|}		
\newcommand{\norm}[1]{\left\|#1\right\|}		
\newcommand{\sprod}[1]{\langle#1\rangle}		
\newcommand{\fall}{\;\text{ for all }}		
\newcommand{\on}{\;\text{ on }}		
\newtheorem{theorem}{Theorem}[section]
\newtheorem{lemma}[theorem]{Lemma}
\newtheorem{corollary}[theorem]{Corollary}
\newtheorem{proposition}[theorem]{Proposition}

\theoremstyle{definition}

\newtheorem{example}[theorem]{Example}
\newtheorem*{hypothesis*}{Hypothesis}
\theoremstyle{remark}
\newtheorem{remark}[theorem]{Remark}
\newcommand{\cref}[1]{Cor.~\ref{#1}}
\newcommand{\eref}[1]{Ex.~\ref{#1}}
\newcommand{\href}[1]{Hyp.~\ref{#1}}
\newcommand{\pref}[1]{Prop.~\ref{#1}}
\newcommand{\tref}[1]{Thm.~\ref{#1}}
\newcommand{\lref}[1]{Lemma~\ref{#1}}
\newcommand{\rref}[1]{Rem.~\ref{#1}}
\allowdisplaybreaks

\numberwithin{equation}{section}

\begin{document}

\title{Evans function, parity and nonautonomous bifurcations}


\author{Christian P\"otzsche}
\address{Christian P\"otzsche, Department of Mathematics, University of Klagenfurt, Universit\"atsstra{\ss}e 65--67, 9020 Klagenfurt, Austria}
\email{christian.poetzsche@aau.at}

\author{Robert Skiba}
\address{Robert Skiba, Faculty of Mathematics and Computer Science, Nicolaus Copernicus University in Toru\'n, ul.\ Chopina 12/18, 87-100 Toru{\'n}, Poland}
\email{robert.skiba@mat.umk.pl}

\subjclass[2020]{Primary 47J15, 34C37, 34C23; Secondary 47A53, 37C60}

\date{}

\dedicatory{}

\begin{abstract}
	The concept of parity due to Fitzpatrick, Pejsachowicz and Rabier is a central tool in the abstract bifurcation theory of nonlinear Fredholm operators. In this paper, we relate the parity to the Evans function, which is widely used in the stability analysis for traveling wave solutions to evolutionary PDEs.
	
	As application we obtain a flexible and general condition yielding local bifurcations of specific bounded entire solutions to (Carath{\'e}odory) differential equations. These bifurcations are intrinsically nonautonomous in the sense that the assumptions implying them cannot be fulfilled for autonomous or periodic temporal forcings. In addition, we demonstrate that Evans functions are strictly related to the dichotomy spectrum and hyperbolicity, which play a crucial role in studying the existence of bounded solutions on the whole real line and therefore the recent field of nonautonomous bifurcation theory. Finally, by means of non-trivial examples we illustrate the applicability of our methods.
\end{abstract}

\maketitle

\section{From Carath{\'e}odory to Krasnoselskii and beyond}
This paper investigates the local behavior of nonautonomous evolutionary differential equations under parameter variation. In contrast to the classical theory of dynamical systems, one cannot expect that such explicitly time-variant problems possess constant solutions (equilibria). For this reason, the recent nonautonomous bifurcation theory investigates changes in the structure of (forward or pullback) attractors or in the set of bounded entire solutions \cite{anagnostopoulou:poetzsche:rasmussen:22}. Apparently both approaches are related because pullback attractors consist of bounded entire solutions.

More detailed, we study parametrized nonautonomous differential equations
\begin{equation}
	\tag{$C_\lambda$}
	\dot x=f(t,x,\lambda)
	\label{cde}
\end{equation}
in $\R^d$ allowing merely measurable dependence on the time variable (one speaks of \emph{Cara\-th{\'e}odory} \emph{equations} \cite{aulbach:wanner:96,kurzweil:86}). These problems naturally occur in the field of Random Dynamical Systems as pathwise realization of random differential equations \cite{arnold:98}, in Control Theory when working with essentially bounded control functions \cite{colonius:kliemann:99}, and clearly include the special case of nonautonomous ordinary differential equations. Aiming to detect bifurcations in Carath{\'e}odory equations \eqref{cde}, our strategy to locate their bounded entire solutions is to characterize them as zeros of an abstract parametrized operator between suitable spaces of bounded functions. This allows to employ corresponding tools from the functional analysis of Fredholm operators. In this setting, both sufficient, but also necessary conditions for local bifurcations of bounded entire solutions were already established in \cite{poetzsche:12} (see also \cite[pp.~42ff]{anagnostopoulou:poetzsche:rasmussen:22}). Nonetheless, although \cite{poetzsche:12} provides a precise information on the local structure of the bifurcating solutions, it is restricted to a particular form of nonyperbolicity and requires specific smoothness and further assumptions on the partial derivatives of $f$.

In contrast to \cite{poetzsche:12}, the contribution at hand is less focussed on a detailed description of bifurcation diagrams. We rather intend to introduce a more general and easily applicable tool to detect changes in the set of bounded entire solutions to \eqref{cde}, when $\lambda$ varies. A starting point for such an endeavor might be the classical result of Krasnoselskii that odd algebraic multiplicity of critical eigenvalues for the linearization of a parametrized nonlinear equation implies bifurcation. This can be seen as an initial contribution to abstract analytical bifurcation theory (cf.\ \cite{krasnosielski:56} or e.g.~\cite[p.~204, Thm.~II.3.2]{kielhoefer:12}). It is nevertheless restricted to nonlinear fixed-point problems involving completely continuous operators. In nonautonomous bifurcation theory the operators characterizing bounded entire solutions to ordinary differential or Carath{\'e}odory equations \eqref{cde} leave this setting. Hence, the Leray--Schauder degree and specifically the classical Krasnoselskii bifurcation theorem cannot be applied. One rather needs a degree theory, a concept of multiplicity and ambient bifurcation results tailor-made for our more general class of nonlinear operators. We demonstrate that the \emph{parity} developed in \cite{Fitz91, fitzPeja86, FiPejsachowiczIII, FitPejRab} is indeed a tool suitable for these purposes. This topological invariant applies to a continuous path of index $0$ Fredholm operators and plays a fundamental role in the degree and abstract bifurcation theory of nonlinear Fredholm mappings \cite{FiPejsachowiczIV, pejsachowicz:11a, pejsachowicz:11b, Pej-Ski} or \cite{esquinas:lopez:88, esquinas:88, lopez:sampedro:23}. Yet, explicit parity computations depend on the particular problems and are nontrivial.

In our situation, Fred\-holm\-ness means that variation equations of \eqref{cde} along continuous families of bounded solutions possess compatible exponential dichotomies on both semiaxes \cite[Lemma~4.2]{palmer:84}. The alert reader might realize that a related constellation is also met in the stability theory for traveling wave solutions (pulses, shock layers) of various types of evolutionary PDEs (see e.g.\ \cite{sandstede:02,kapitula:promislow:13}). In this area the \emph{Evans function} is a complex-valued analytical function, whose set of zeros coincides with the point spectrum of a differential operator arising as linearization along the wave. The order of the zero gives the algebraic multiplicity of the eigenvalues and based on the Argument Principle one even obtains information on the total number of zeros. Thus, the Evans function is of crucial importance in this field and allows explicit computations.

For the sake of nonautonomous bifurcation theory we associate real Evans functions $E$ to variation equations of \eqref{cde} along a given continuous family of bounded entire solutions. It suffices to demand that $E$ is continuous in the real bifurcation parameter $\lambda$. Now the benefit of an Evans function is twofold: First, their zeros indicate parameters where critical intervals of the dichotomy spectrum \cite{siegmund:02} split into a hyperbolic situation (cf.\ \cref{corEvans}). Second, as our essential contribution we establish in \tref{thmmain} that the parity of a path of Fredholm operators can be expressed as product of the signs of Evans functions evaluated at the boundary points of the path. Hence, based on an abstract bifurcation result culminating from \cite{Fitz91, FiPejsachowiczIV, pejsachowicz:11a, pejsachowicz:11b} it results that a sign change of $E$ is even sufficient for a whole continuum of bounded entire solutions to bifurcate. Referring to \cite{alexander:gardner:jones:90} we note that the parity is not the only topological invariant associated to the Evans function and point out that $E$ can be numerically approximated \cite{dieci:elia:vanvleck:11}.

This paper is structured as follows. The subsequent Sect.~\ref{sec2} contains necessary basics on Ca\-ra\-th{\'e}o\-dory equations \eqref{cde} and introduces an abstract parametrized operator (between spaces of essentially bounded functions), whose zeros characterize the bounded solutions of \eqref{cde}. Based on exponential dichotomy assumptions for a variation equation associated to \eqref{cde} we establish that this operator is Fredholm. Here large parts of the required Fredholm theory are admittedly akin to results for ordinary differential equations due to \cite{coppel:78,palmer:84,palmer:88}, but also for the sake of later reference in this text, we provide rather detailed proofs. Then an Evans function tailor-made for our bifurcation theory is introduced and studied in Sect.~\ref{sec3}, which results in the crucial \tref{thmmain} relating parity and Evans function. As application, Sect.~\ref{sec4} features a rather general sufficient condition for the bifurcation of bounded solutions to Carath{\'e}odory equations \eqref{cde} from a prescribed branch $\phi_\lambda$ in \tref{thmsingle}. The solutions contained in this bifurcating continuum are in fact perturbations of the $\phi_\lambda$ vanishing at $t=\pm\infty$ (one speaks of \emph{homoclinic} solutions). This bifurcation criterion is illustrated by means of two concrete examples, where the first one involves a Fredholm operator of arbitrary kernel dimension. An outlook to the scope of our approach is given in Sect.~\ref{sec5}. Finally, for the convenience of the reader, App.~\ref{appA} describes constructions of the parity and its properties, while App.~\ref{appB} presents an abstract bifurcation result suitable for applications to \eqref{cde}.
\paragraph{Notation}
We write $\R_+:=[0,\infty)$, $\R_-:=(-\infty,0]$ for the semiaxes and $\delta_{ij}$ for the Kronecker symbol. The interior and boundary of a subset $\Lambda$ of a metric space are denoted by $\Lambda^\circ$ resp.\ $\partial\Lambda$; the distance of a point $x$ to the set $\Lambda$ is $\dist_\Lambda(x):=\inf_{\lambda\in\Lambda}d(x,\lambda)$.

If $X,Y$ are Banach spaces, then $L(X,Y)$ are the linear bounded, $GL(X,Y)$ are the bounded invertible and $F_0(X,Y)$ abbreviate the index $0$ Fredholm operators from $X$ to $Y$; $I_X$ is the identity map on $X$, $N(T)$ the kernel and $R(T)$ the range of $T\in L(X,Y)$.

On the Euclidean space $\R^d$ we employ the canonical unit vectors $e_i:=(\delta_{ij})_{j=1}^d$ for $1\leq i\leq d$, the inner product $\sprod{x,y}:=\sum_{j=1}^dx_jy_j$ with induced norm $\abs{x}:=\sqrt{\sprod{x,x}}$ and denote the orthogonal complement of a subspace $V\subseteq\R^d$ by $V^\perp$. Moreover, $I_d$ and $0_d$ is the identity resp.\ zero matrix in $\R^{d\tm d}$, and $A^T$ is the transpose of a matrix $A\in\R^{d\tm d}$. We equip $\R^{d\tm d}$ with the norm induced by the Euclidean norm $\abs{\cdot}$.

Given a function $\phi:\R\to\R^d$ and $\rho>0$ we define the open $\rho$-neighborhood of its graph as
$
	\cB_\rho(\phi):=\set{(t,x)\in\R\tm\R^d:\,\abs{x-\phi(t)}<\rho}.
$
\section{Carath{\'e}odory equations and Fredholm theory}
\label{sec2}
Let $\Omega\subseteq\R^d$ be nonempty, open, convex and assume $(\tilde\Lambda,d)$ is a metric space. Our investigations center around parameter-dependent Carath{\'e}odory equations
\begin{equation}
	\tag{$C_\lambda$}
	\dot x=f(t,x,\lambda),
\end{equation}
whose right-hand side $f:\R\tm\Omega\tm\tilde\Lambda\to\R^d$ is a \emph{Carath{\'e}odory function}, i.e.\ for every parameter value $\lambda\in\tilde\Lambda$ and
\begin{itemize}
	\item for every $x\in\Omega$ the mapping $f(\cdot,x,\lambda):\R\to\R^d$ is measurable,

	\item for almost every $t\in\R$ the mapping $f(t,\cdot,\lambda):\Omega\to\R^d$ is continuous.
\end{itemize}
Throughout, measurability and integrability are understood in the Lebesgue sense. More precisely, we work under the following standing assumptions:
\begin{hypothesis*}[$\mathbf{H_0}$] The right-hand side $f:\R\tm\Omega\tm\tilde\Lambda\to\R^d$ of \eqref{cde} is a Carath{\'e}odory function with the following properties: For almost every $t\in\R$ and each $\lambda\in\tilde\Lambda$ the function $f(t,\cdot,\lambda):\Omega\to\R^d$ is differentiable with continuous partial derivative $D_2f(t,\cdot):\Omega\tm\tilde\Lambda\to\R^{d\tm d}$ such that for all bounded $B\subseteq\Omega$ and all $j\in\set{0,1}$, one has
	\begin{equation}
		\esssup_{t\in\R}\sup_{x\in B}\abs{D_2^jf(t,x,\lambda)}<\infty
		\fall\lambda\in\tilde\Lambda.
		\label{nostar1}
	\end{equation}
	Moreover, for each $\lambda_0\in\tilde\Lambda$ and $\eps>0$ there exists a $\delta>0$ with
	\begin{equation*}
		\abs{x-y}<\delta
		\quad\Rightarrow\quad
		\esssup_{t\in\R}\abs{D_2^jf(t,x,\lambda)-D_2^jf(t,y,\lambda_0)}<\eps
	\end{equation*}
	for all $x,y\in\Omega$ and $\lambda\in B_\delta(\lambda_0)$.
\end{hypothesis*}

Keeping $\lambda\in\tilde\Lambda$ fixed, a \emph{solution} to \eqref{cde} is a continuous function $\phi:I\to\Omega$ defined on an interval $I\subseteq\R$ satisfying the Volterra integral equation (cf.~\cite[Def.~2.3]{aulbach:wanner:96})
\begin{equation}
	\phi(t)=\phi(\tau)+\int_\tau^tf(s,\phi(s),\lambda)\d s\fall\tau,t\in I.
	\label{nosol}
\end{equation}
In case $I=\R$ one speaks of an \emph{entire solution} and then we denote $\phi$ as \emph{permanent}, provided $\inf_{t\in\R}\dist_{\partial\Omega}(\phi(t))>0$ holds, that is, the solution values $\phi(t)$ keep a positive distance from the boundary of $\Omega$. We denote the unique solution to \eqref{cde} satisfying the initial condition $x(\tau)=\xi$ as \emph{general solution} $\varphi_\lambda(\cdot;\tau,\xi)$, where $(\tau,\xi)\in\R\tm\Omega$.

\begin{hypothesis*}[$\mathbf{H_1}$] The Carath{\'e}odory equation \eqref{cde} has a family $(\phi_\lambda)_{\lambda\in\tilde\Lambda}$ of bounded permanent solutions $\phi_\lambda:\R\to\Omega$ such that for every $\eps>0$, $\lambda_0\in\tilde\Lambda$ there is a $\delta>0$ with
	\begin{equation*}
		d(\lambda,\lambda_0)<\delta
		\quad\Rightarrow\quad
		\sup_{t\in\R}\abs{\phi_\lambda(t)-\phi_{\lambda_0}(t)}<\eps\fall\lambda\in\tilde\Lambda
	\end{equation*}
	and there exists a $\bar\rho>0$ with $\inf_{t\in\R}\dist_{\partial\Omega}(\phi_\lambda(t))>\bar\rho$ for all $\lambda\in\tilde\Lambda$.
\end{hypothesis*}
In this context, an entire solution $\phi:\R\to\Omega$ to \eqref{cde} is called \emph{homoclinic} to $\phi_\lambda$, provided the limit relations $\lim_{t\to\pm\infty}\abs{\phi(t)-\phi_\lambda(t)}=0$ hold.

Central parts of our theory are based on linearization. This involves the \emph{variation equations} corresponding to the solution family $(\phi_\lambda)_{\lambda\in\tilde\Lambda}$ given by
\begin{align}
	\tag{$V_\lambda$}
	\dot x&=A(t,\lambda)x,&
	A(t,\lambda)&:=D_2f(t,\phi_\lambda(t),\lambda)
	\label{var}
\end{align}
with coefficient matrices $A:\R\tm\tilde\Lambda\to\R^{d\tm d}$ having the immediate properties:
\begin{lemma}\label{lemaprop}
	If $(H_0$--$H_1)$ hold, then
	\begin{enumerate}
		\item[$(a)$] $A(\cdot,\lambda):\R\to\R^{d\tm d}$ is essentially bounded and thus locally integrable for $\lambda\in\tilde\Lambda$,

		\item[$(b)$] $A(t,\cdot):\tilde\Lambda\to\R^{d\tm d}$ is continuous for a.a.\ $t\in\R$.
	\end{enumerate}
\end{lemma}
Hence, the \emph{transition matrix} $\Phi_\lambda(t,s)\in GL(\R^d,\R^d)$, $t,s\in\R$, of \eqref{var} is well-defined and due to \cite[Lemma~2.9]{aulbach:wanner:96} of \emph{bounded growth}, i.e.\
\begin{equation*}
	\abs{\Phi_\lambda(t,s)}\leq\exp\intoo{\esssup_{r\in\R}\abs{A(r,\lambda)}\abs{t-s}}\fall s,t\in\R,\,\lambda\in\tilde\Lambda.
\end{equation*}

For $\lambda\in\tilde\Lambda$ fixed again, a solution $\phi_\lambda:\R\to\Omega$ is understood as \emph{hyperbolic} on a subinterval $I\subseteq\R$, if the associated variation equation \eqref{var} is \emph{exponentially dichotomic} on~$I$. This means there exist reals $K\geq 1$, \emph{growth rates} $\alpha>0$ and a projection-valued function $P_{\lambda}:I\to\R^{d\tm d}$ such that
\begin{equation}
	\Phi_\lambda(t,s)P_\lambda(s)=P_\lambda(t)\Phi_\lambda(t,s)
	\label{ed0}
\end{equation}
(one speaks of an \emph{invariant projector}) and
\begin{align}
	\abs{\Phi_\lambda(t,s)P_\lambda(s)}&\leq Ke^{-\alpha(t-s)},&
	\abs{\Phi_\lambda(s,t)[I_d-P_\lambda(t)]}&\leq Ke^{-\alpha(t-s)}
	\label{ed1}
\end{align}	
for all $s\leq t$, $t,s\in I$. The \emph{dichotomy spectrum} of \eqref{var} is given by (cf.~\cite{siegmund:02})
\begin{equation*}
	\Sigma(\lambda)
	:=
	\set{\gamma\in\R:\,\dot x=[A(t,\lambda)-\gamma I_d]x\text{ has no exponential dichotomy on }\R}
\end{equation*}
and consist of $d_0\in\set{1,\ldots,d}$ compact \emph{spectral intervals} $\sigma_j\subseteq\R$, i.e.\ $\Sigma(\lambda)=\bigcup_{j=1}^{d_0}\sigma_j$ (cf.~\cite[Thm.~3.1]{siegmund:02}). To each $\sigma_j$ one associates a spectral manifold $\cV_j\subseteq\R\tm\R^d$, which is an invariant bundle of subspaces of $\R^d$ having constant dimension called \emph{algebraic multiplicity} $\mu_j\in\set{1,\ldots,d}$ of the spectral interval $\sigma_j$, $1\leq j\leq d_0$.

Note that the ranges $R(P_\lambda(\tau))$, $\tau\in I$, are uniquely determined on intervals $I$ unbounded above, while the kernels $N(P_\lambda(\tau))$, $\tau\in I$, are unique on intervals $I$ unbounded below. Given this, we introduce the \emph{Morse index} (note that it is independent of $\tau\in I$)
\begin{equation*}
	m_\lambda
	:\equiv
	\begin{cases}
		\dim N(P_\lambda(\tau)),&\text{ if $I$ is unbounded below},\\
		d-\dim R(P_\lambda(\tau)),&\text{ if $I$ is unbounded above}.
	\end{cases}
\end{equation*}
In particular, one has the respective dynamical characterizations (cf.~\cite[p.~19]{coppel:78})
\begin{align}\label{R-N}
	\begin{split}
		R(P_\lambda(\tau))
		&=
		\set{\xi\in\R^d:\,\sup_{\tau\leq t}e^{\gamma(\tau-t)}|\Phi_\lambda(t,\tau)\xi|<\infty}
		\fall\gamma\in[-\alpha,\alpha),\\
		N(P_\lambda(\tau))
		&=
		\set{\xi\in\R^d:\,\sup_{t\leq\tau}e^{\gamma(\tau-t)}|\Phi_\lambda(t,\tau)\xi|<\infty}
		\fall\gamma\in(-\alpha,\alpha]
	\end{split}
\end{align}
and $\tau\in I$. Eventually, it is convenient to introduce the \emph{Green's function}
\begin{equation}
	\Gamma_{P_\lambda}(t,s)
	:=
	\begin{cases}
		\Phi_\lambda(t,s)P_\lambda(s),&s\leq t,\\
		-\Phi_\lambda(t,s)[I_d-P_\lambda(s)],&t<s
	\end{cases}
	\fall s,t\in I.
	\label{nogreen}
\end{equation}

Our functional analytical approach requires a suitable setting of functions defined on an interval $I\subseteq\R$. We write $L^\infty(I,\Omega)$ for the essentially bounded and $W^{1,\infty}(I,\Omega)$ for $L^\infty$-functions $x: I\to\Omega$ with essentially bounded (weak) derivatives. In case $\Omega=\R^d$ we write $L^\infty(I):=L^\infty(I,\R^d)$ and proceed accordingly with further function spaces. Note that $L^\infty(I)$ is a Banach space w.r.t.\ the norm
$
	\norm{x}_\infty:=\esssup_{t\in I}\abs{x(t)}.
$

Each $x\in W^{1,\infty}(I)$ has a bounded Lipschitz continuous representative (cf.\ \cite[p.~224, Thm.~7.17]{leoni:09}), while Rademacher's theorem \cite[p.~343, Thm.~11.49]{leoni:09} yields that the (strong) derivative $\dot x:I\to\R^d$ exists a.e.\ in $I\subseteq\R$. From \cite[p.~224, Ex.~7.18]{leoni:09} we obtain that $W^{1,\infty}(I)$ is a Banach space with
$
	\norm{x}_{1,\infty}:=\max\set{\norm{x}_\infty,\norm{\dot x}_\infty}
$
as norm.

Clearly, $W^{1,\infty}(I)\subseteq L^\infty(I)$ is a continuous embedding. Finally, on the interval $I=\R$ and for $0\in\Omega$ we introduce the respective subsets
\begin{align*}
	L_0^\infty(\R,\Omega)
	&:=
	\set{x\in L^\infty(\R,\Omega)\mid\forall\eps>0\exists T>0:\,\abs{x(t)}<\eps\text{ a.e.\ in }\R\setminus(-T,T)},\\
	W_0^{1,\infty}(\R,\Omega)
	&:=
	\set{x\in W^{1,\infty}(\R,\Omega)\mid x,\dot x\in L_0^\infty(\R)}.
\end{align*}
Then the continuous embeddings $W^{1,\infty}(\R)\subseteq L^\infty(\R)$ and $W_0^{1,\infty}(\R)\subseteq L_0^\infty(\R)$ hold.

We characterize bounded entire solutions of Carath{\'e}odory equations \eqref{cde}, as well as solutions being homoclinic to the family $\phi_\lambda$ from Hypothesis $(H_1)$, as zeros
\begin{equation}
	\tag{$O_\lambda$}
	G(x,\lambda)=0
	\label{abs}
\end{equation}
of the formally defined abstract nonlinear operator
\begin{equation*}
	[G(x,\lambda)](t):=\dot x(t)-f(t,x(t)+\phi_\lambda(t),\lambda)+f(t,\phi_\lambda(t),\lambda).
\end{equation*}
One clearly obtains the identity $G(0,\lambda)\equiv 0$ on $\tilde\Lambda$.
\begin{theorem}\label{thmgprop}
	If $(H_0$--$H_1)$ hold, then
	\begin{enumerate}
		\item[$(a)$] $G:U\tm\tilde\Lambda\to L^\infty(\R)$ is well-defined on $U:=\bigl\{x\in W^{1,\infty}(\R):\,\norm{x}_\infty<\bar\rho\bigr\}$, con\-ti\-nuous and the partial derivative $D_1G:U^\circ\tm\tilde\Lambda\to L(W^{1,\infty}(\R),L^\infty(\R))$ exists as a continuous function,

		\item[$(b)$] $G:U\tm\tilde\Lambda\to L_0^\infty(\R)$ is well-defined on $U:=\bigl\{x\in W_0^{1,\infty}(\R):\,\norm{x}_\infty<\bar\rho\bigr\}$, con\-tinuous and the partial derivative $D_1G:U^\circ\tm\tilde\Lambda\to L(W_0^{1,\infty}(\R),L_0^\infty(\R))$ exists as a continuous function.
	\end{enumerate}
	Moreover, in both cases and for $\lambda\in\tilde\Lambda$ the partial derivative is given by
	\begin{equation}
		[D_1G(x,\lambda)y](t)=\dot y(t)-D_2f(t,x(t)+\phi_\lambda(t),\lambda)y(t)
		\quad\text{for a.a.\ }t\in\R.
		\label{thmgprop1}
	\end{equation}
\end{theorem}
\begin{proof}
	The argument essentially follows \cite[Cor.~2.1]{poetzsche:12}.
\end{proof}

\begin{theorem}\label{thmchar}
	If $(H_0$--$H_1)$ hold, then $\phi_\lambda\in W^{1,\infty}(\R,\Omega)$ and also the following is true for all parameters $\lambda\in\tilde\Lambda$:
	\begin{enumerate}
		\item[$(a)$] If $\phi:\R\to\Omega$ is a bounded solution of \eqref{cde} in $\cB_{\bar\rho}(\phi_\lambda)$, then $\phi-\phi_\lambda$ is contained in $W^{1,\infty}(\R)$ and satisfies \eqref{abs}. Conversely, if $\psi\in L^\infty(\R)$ has a (strong) derivative a.e.\ in $\R$ with $\norm{\psi}_\infty<\bar\rho$ and satisfies $G(\psi,\lambda)=0$, then $\psi\in W^{1,\infty}(\R)$ and $\psi+\phi_\lambda$ is a bounded entire solution of \eqref{cde} in $\cB_{\bar\rho}(\phi_\lambda)$.

		\item[$(b)$] If $\phi:\R\to\Omega$ is a solution of \eqref{cde} in $\cB_{\bar\rho}(\phi_\lambda)$ homoclinic to $\phi_\lambda$, then $\phi-\phi_\lambda$ is contained in $W_0^{1,\infty}(\R)$ and satisfies \eqref{abs}. Conversely, if $\psi\in L_0^\infty(\R)$ has a (strong) derivative a.e.\ in $\R$ with $\norm{\psi}_\infty<\bar\rho$ and satisfies $G(\psi,\lambda)=0$, then $\psi\in W_0^{1,\infty}(\R)$ and $\psi+\phi_\lambda$ is a solution of \eqref{cde} in $\cB_{\bar\rho}(\phi_\lambda)$ homoclinic to $\phi_\lambda$.
	\end{enumerate}
\end{theorem}
\begin{proof}
	Let $\lambda\in\tilde\Lambda$ be fixed. The assumption $(H_1)$ directly yields $\phi_\lambda\in L^\infty(\R,\Omega)$. As a solution to \eqref{cde}, $\phi_\lambda$ is absolutely continuous, hence the strong derivative $\dot\phi_\lambda$ exists a.e.\ in $\R$ and satisfies $\dot\phi_\lambda(t)\equiv f(t,\phi_\lambda(t),\lambda)$. Thus, since \eqref{nostar1} yields that $f(\cdot,\phi_\lambda(\cdot),\lambda)$ is essentially bounded on $\R$ and we deduce $\dot\phi_\lambda\in L^\infty(\R)$, i.e.\ $\phi_\lambda\in W^{1,\infty}(\R,\Omega)$.

	(a) Assume $\phi\in L^\infty(\R,\Omega)$ is an entire solution of \eqref{cde} in $\cB_{\bar\rho}(\phi_\lambda)$. Then $\phi$ and $\phi_\lambda$ both satisfy \eqref{nosol} and the Fundamental Theorem of Calculus \cite[p.~85, Thm.~3.30]{leoni:09} yields that $\phi,\phi_\lambda$ are absolutely continuous (on any bounded subinterval of $\R$). This yields that the strong derivatives $\dot\phi$, $\dot\phi_\lambda$ exist a.e.\ in $\R$. Thus, $\delta:=\phi-\phi_\lambda\in L^\infty(\R)$ fulfills the identity $\dot\delta(t)+\dot\phi_\lambda(t)\equiv f(t,\delta(t)+\phi_\lambda(t),\lambda)$ a.e.\ on $\R$. Hence, $\norm{\phi-\phi_\lambda}_\infty<\bar\rho$ and the fact
	$
		\dot\delta(t)
		\equiv
		f(t,\delta(t)+\phi_\lambda(t),\lambda)-f(t,\phi_\lambda(t),\lambda)
	$
	a.e.\ on $\R$ has two consequences: First, there exists a bounded set $B\subseteq\Omega$ such that the inclusion $\phi_\lambda(t)+\theta\delta(t)\in B$ holds for all $t\in\R$, $\theta\in[0,1]$ due to the convexity of $\Omega$. Whence the Mean Value Theorem \cite[p.~243, Thm.~4.C for $n=1$]{zeidler:95} implies
	\begin{align*}
		\abs{\dot\delta(t)}
		&=
		\abs{\int_0^1D_2f(t,\theta\delta(t)+\phi_\lambda(t),\lambda)\d\theta\delta(t)}\\
		&\leq
		\int_0^1\abs{D_2f(t,\theta\delta(t)+\phi_\lambda(t),\lambda)\d\theta}\norm{\delta}_\infty
	\end{align*}
	and thanks to $(H_0)$ the right-hand side of this inequality is essentially bounded in $t\in\R$, i.e.\ $\delta\in W^{1,\infty}(\R)$ holds. Second, $\delta$ defines an entire solution of the equation of perturbed motion $\dot x=f(t,x+\phi_\lambda(t),\lambda)-f(t,\phi_\lambda(t),\lambda)$, which in turn implies $G(\delta,\lambda)=0$.

	Conversely, let $\psi\in L^\infty(\R)$ be strongly differentiable a.e.\ in $\R$ with $\norm{\psi}_\infty<\bar\rho$ and $G(\psi,\lambda)=0$, i.e.\ $\dot\psi(t)=f(t,\psi(t)+\phi_\lambda(t),\lambda)-f(t,\phi_\lambda(t),\lambda)$ holds for a.a.\ $t\in\R$. First, $\dot\psi(t)+\dot\phi_\lambda(t)=f(t,\psi(t)+\phi_\lambda(t),\lambda)$ a.e.\ in $\R$ implies that $\psi+\phi_\lambda$ is a bounded entire solution of \eqref{cde} in $\cB_{\bar\rho}(\phi_\lambda)$. Second, as above one establishes $\dot\psi\in L^\infty(\R)$ and therefore the inclusion $\psi\in W^{1,\infty}(\R)$ results.

	(b) can be shown analogously.
\end{proof}

\begin{theorem}[admissibility]\label{thmadmin}
	If $(H_0$--$H_1)$ hold, then the following are equivalent for all parameters $\lambda\in\tilde\Lambda$:
	\begin{enumerate}
		\item[$(a)$] $D_1G(0,\lambda)\in GL(W^{1,\infty}(\R),L^\infty(\R))$,

		\item[$(b)$] $D_1G(0,\lambda)\in GL(W_0^{1,\infty}(\R),L_0^\infty(\R))$,

		\item[$(c)$] the bounded entire solution $\phi_\lambda:\R\to\Omega$ to \eqref{cde} is hyperbolic on $\R$.
	\end{enumerate}
\end{theorem}
\begin{proof}
	Throughout, let $\lambda\in\tilde\Lambda$ be fixed.

	$(c)\Rightarrow(a)$ Because $\phi_\lambda$ is hyperbolic, \eqref{var} has an exponential dichotomy on $\R$ with projector $P_\lambda$ and growth rate $\alpha>0$. Due to the explicit form \eqref{thmgprop1} from \tref{thmgprop} the invertibility of the Fr{\'e}chet derivative $D_1G(0,\lambda)$ means that for each $g\in L^\infty(\R)$ there exists a unique solution $\psi\in W^{1,\infty}(\R)$ of the perturbed variation equation
	\begin{equation}
		\tag{$V_{\lambda,g}$}
		\dot x=A(t,\lambda)x+g(t).
		\label{linb}
	\end{equation}
	In order to verify this, with Green's function \eqref{nogreen} we define
	\begin{align*}
		\psi:\R&\to\R^d,&
		\psi(t)&:=\int_\R\Gamma_{P_\lambda}(t,s)g(s)\d s.
	\end{align*}
	As in \cite[proof of Lemma~3.2]{aulbach:wanner:96} one establishes that $\psi$ is actually a solution of \eqref{linb}. Moreover, the dichotomy estimates \eqref{ed1} yield $\psi\in L^\infty(\R)$. Hence, since the solution identity $\dot\psi(t)\equiv A(t,\lambda)\psi(t)+g(t)$ holds a.e.\ in $\R$ we obtain from \lref{lemaprop} and the inclusion $g\in L^\infty(\R)$ that also $\dot\psi$ is essentially bounded, i.e.\ $\psi\in W^{1,\infty}(\R)$. It remains to show that $\psi$ is uniquely determined by the inhomogeneity $g\in L^\infty(\R)$. If $\bar\psi\in L^\infty(\R)$ is another bounded entire solution to \eqref{linb}, then the difference $\psi-\bar\psi\in L^\infty(\R)$ solves the variation equation \eqref{var}. Due to our hyperbolicity assumption, \eqref{var} has exponential dichotomies on $\R_+$ and on $\R_-$ with projector $P_\lambda$ satisfying $\R^d=R(P_\lambda(0))\oplus N(P_\lambda(0))$. This yields $R(P_\lambda(0))\cap N(P_\lambda(0))=\set{0}$ and the dynamical characterization \eqref{R-N} implies that the trivial solution is the unique bounded entire solution to \eqref{var}; thus $\psi=\bar\psi$. This shows that $D_1G(0,\lambda):W^{1,\infty}(\R)\to L^\infty(\R)$ is invertible and Banach's Isomorphism Theorem \cite[pp.~179--180, Prop.~1]{zeidler:95} yields the claim.

	$(a)\Rightarrow(b)$ Repeating the arguments yielding (a) it remains to show that inhomogeneities $g\in L_0^\infty(\R)$ imply $\psi\in W_0^{1,\infty}(\R)$. Thereto, note $\psi=\psi_1+\psi_2$ with
	\begin{align*}
		\psi_1(t)&:=\int_{-\infty}^t\Phi_\lambda(t,s)P_\lambda(s)g(s)\d s,&
		\psi_2(t)&:=-\int_t^{\infty}\Phi_\lambda(t,s)[I_d-P_\lambda(s)]g(s)\d s
	\end{align*}
	and choose $\eps>0$. Then there exists a real $T>0$ such that $\abs{g(s)}<\tfrac{\alpha}{K}\tfrac{\eps}{2}$ holds for a.a.\ $s\in\R\setminus(-T,T)$. On the one hand, provided we choose $T_1\geq T$ sufficiently large that $\tfrac{K}{\alpha}\norm{g}_\infty e^{\alpha(T-t)}<\tfrac{\eps}{2}$ for all $t\geq T_1$, then this leads to the estimates
	\begin{align*}
		\abs{\psi_1(t)}
		&\leq
		\int_{-\infty}^T\abs{\Phi_\lambda(t,s)P_\lambda(s)g(s)\d s}+\int_T^\infty\abs{\Phi_\lambda(t,s)P_\lambda(s)g(s)\d s}\\
		&\stackrel{\eqref{ed1}}{\leq}
		K\int_{-\infty}^Te^{-\alpha(t-s)}\abs{g(s)}\d s+K\int_T^\infty e^{-\alpha(t-s)}\abs{g(s)}\d s\\
		&\leq
		K\int_{-\infty}^Te^{-\alpha(t-s)}\d s\norm{g}_\infty+\alpha\int_T^\infty e^{-\alpha(t-s)}\d s\frac{\eps}{2}\\
		&\leq
		\frac{K}{\alpha}e^{\alpha(T-t)}\norm{g}_\infty+\frac{\eps}{2}
		<
		\eps\fall t\geq T_1,\\
		\abs{\psi_1(t)}
		&\leq
		\int_{-\infty}^t\abs{\Phi_\lambda(t,s)P_\lambda(s)g(s)}\d s
		\stackrel{\eqref{ed1}}{\leq}
		\int_{-\infty}^te^{-\alpha(t-s)}\d s\eps
		=
		\eps\fall t\leq-T
	\end{align*}
	and in conclusion $\lim_{t\to\pm\infty}\abs{\psi_1(t)}=0$. On the other hand, if we furthermore choose $T_1\geq T$ so large that $\frac{K}{\alpha}\norm{g}_\infty e^{\alpha(t+T)}<\frac{\eps}{2}$ for all $t\leq-T_1$, then
	\begin{align*}
		\abs{\psi_2(t)}
		&\leq
		\int_t^{\infty}\abs{\Phi_\lambda(t,s)[I_d-P_\lambda(s)]g(s)}\d s
		\stackrel{\eqref{ed1}}{\leq}
		\alpha\int_t^\infty e^{\alpha(t-s)}\eps
		=
		\eps
		\fall t\geq T,\\
		\abs{\psi_2(t)}
		&\leq
		\int_t^{-T}\abs{\Phi_\lambda(t,s)[I_d-P_\lambda(s)]g(s)}\d s+
		\int_{-T}^\infty\abs{\Phi_\lambda(t,s)[I_d-P_\lambda(s)]}\d s\\
		&\stackrel{\eqref{ed1}}{\leq}
		K\int_{-T}^\infty e^{\alpha(t-s)}\abs{g(s)}\d s+K\int_{-T}^\infty e^{\alpha(t-s)}\abs{g(s)}\d s\\
		&\leq
		\alpha\int_{-T}^\infty e^{\alpha(t-s)}\d s\frac{\eps}{2}+K\int_{-T}^\infty e^{\alpha(t-s)}\d s\norm{g}_\infty\\
		&\leq
		\frac{\eps}{2}+\frac{K}{\alpha}e^{\alpha(t+T)}\norm{g}_\infty
		<
		\eps
		\fall t\leq-T
	\end{align*}
	also guarantee $\lim_{t\to\pm\infty}\abs{\psi_2(t)}=0$. We conclude that $\psi\in L_0^\infty(\R)$ holds. Moreover, from the identity $\dot\psi(t)\equiv A(t,\lambda)\psi(t)+g(t)$ a.e.\ on $\R$, \lref{lemaprop}(a) and $g\in L_0^\infty(\R)$ results $\lim_{t\to\pm\infty}|\dot\psi(t)|=0$ and consequently $\psi\in W_0^{1,\infty}(\R)$.

	$(b)\Rightarrow(c)$ We first note that the transition matrix $\Phi_\lambda:\R\tm\R\to\R^{d\tm d}$ of \eqref{var} is an evolution family on the Banach space $\R^d$ in the language of e.g.\ \cite{sasu:07,sasu:10} (note that the essential boundedness guaranteed by \lref{lemaprop}(a) and \cite[Lemma~2.9]{aulbach:wanner:96} yield that there exists a $\omega\geq 0$ so that $\abs{\Phi_\lambda(t,s)}\leq e^{\omega(t-s)}$ for all $s\leq t$). In particular, the proof of \cite[Thm.~4.8]{sasu:07} establishes that for each $g\in L_0^\infty(\R)$ there exists a unique solution $\phi\in L_0^\infty(\R)$ of the integral equation
	\begin{equation}
		\phi(t)=\Phi_\lambda(t,\tau)\phi(\tau)+\int_\tau^t\Phi_\lambda(t,s)g(s)\d s\fall\tau\leq t.
		\label{voc1}
	\end{equation}
	Indeed, the abstract setting of \cite[Thm.~1.2]{sasu:10} is met, because the function space $L_0^\infty(\R)$ possesses the following properties:
	\begin{itemize}
		\item For each $x\in L_0^\infty(\R)$ and $\tau\in\R$ the shifted function $x_\tau:=x(\tau+\cdot)$ satisfies $x_\tau\in L_0^\infty(\R)$ and $\norm{x}_\infty=\norm{x_\tau}_\infty$,

		\item $L_0^\infty(\R)$ contains the continuous functions $x:\R\to\R^d$ having compact support,

		\item $\int_\tau^t\abs{x(s)}\d s\leq\int_\tau^t\norm{x}_\infty\d s=(t-\tau)\norm{x}_\infty$ for all $\tau\leq t$ and $x\in L_0^\infty(\R)$,

		\item e.g.\ $x_0:\R\to\R$, $x_0(t):=\tfrac{1}{1+\abs{t}}$ is continuous with $x_0\in L_0^\infty(\R)\setminus L^1(\R)$,

		\item if $x,y:\R\to\R$ are measurable with $\abs{x(t)}\leq\abs{y(t)}$ for a.a.\ $t\in\R$ and $y\in L_0^\infty(\R)$, then $x\in L_0^\infty(\R)$.
	\end{itemize}
	It remains to show $\phi\in W_0^{1,\infty}(\R)$. Thereto, by the Variation of Constants \cite[Thm.~2.10]{aulbach:wanner:96} the unique solution $\phi\in L_0^\infty(\R)$ of \eqref{voc1} is also the unique solution to the perturbed variation equation \eqref{linb} (satisfying the initial condition $x(\tau)=\phi(\tau)$) and hence absolutely continuous on each bounded subinterval of $\R$. Moreover, the solution identity for \eqref{linb} implies $\dot\phi\in L_0^\infty(\R)$ due to \lref{lemaprop}(a). In conclusion, $\phi\in W_0^{1,\infty}(\R)$.
\end{proof}

\begin{lemma}[dual variation equation]
	\label{lemad}
	Let $(H_0$--$H_1)$ hold and $\lambda\in\tilde\Lambda$. If $\phi_\lambda:\R\to\Omega$ is hyperbolic on an interval $I$ with projector $P_\lambda$, then the \emph{dual variation equation}
	\begin{equation}
		\tag{$V_\lambda^\ast$}
		\dot x=-D_2f(t,\phi_\lambda(t),\lambda)^Tx
		\label{varad}
	\end{equation}
	has an exponential dichotomy on $I$ with projector
	\begin{equation}
		Q_\lambda(t):=I_d-P_\lambda(t)^T\fall t\in I.
		\label{proad}
	\end{equation}
In particular, one has $N(P_{\lambda}(t))^{\bot}=N(Q_{\lambda}(t))$ for all $t\in I$.
\end{lemma}
\begin{proof}
	Fix $\lambda\in\tilde\Lambda$. Above all, one can easily show that the dual variation equation \eqref{varad} has the transition matrix $\Phi_\lambda^\ast(t,s):=\phi_\lambda(s,t)^T$ for all $t,s\in I$. Consequently,
\begin{align*}
		\Phi_\lambda^\ast(t,s)Q_\lambda(s)
		&\stackrel{\eqref{proad}}{=}
		\phi_\lambda(s,t)^T\intcc{I_d-P_{\lambda}(s)^T}
		\stackrel{\eqref{ed0}}{=}
		(\intcc{I_d-P_{\lambda}(s)}\phi_\lambda(s,t))^T\\
		&=\left(\phi_\lambda(s,t)[I_d-P_{\lambda}(t)]\right)^T
		=
		\intcc{I_d-P_{\lambda}(t)^T}\phi_\lambda(s,t)^T
		\stackrel{\eqref{proad}}{=}
		Q_{\lambda}(t)\Phi_\lambda^\ast(t,s),
	\end{align*}
	as well as the required dichotomy estimates (note $\abs{C}=\abs{C^T}$ for $C\in\R^{d\tm d}$)
	\begin{align*}
		\left|\Phi_\lambda^\ast(t,s)Q_\lambda(s)\right|
		\stackrel{\eqref{proad}}{=}
		\left|\left(\phi_\lambda(s,t)\intcc{I_d-P_{\lambda}(t)}\right)^T\right|
		=
		\left|\phi_\lambda(s,t)\intcc{I_d-P_{\lambda}(t)}\right|
		\stackrel{\eqref{ed1}}{\leq}
		Ke^{-\alpha(t-s)},\\
		\left|\Phi_\lambda^\ast(s,t)\intcc{I_d-Q_\lambda(t)}\right|
		\stackrel{\eqref{proad}}{=}
		\left|\left(\phi_\lambda(t,s)P_{\lambda}(s)\right)^T\right|
		=
		\left|\phi_\lambda(t,s)P_{\lambda}(s)\right|
		\stackrel{\eqref{ed1}}{\leq}
		Ke^{-\alpha(t-s)}
	\end{align*}
	for all $s\leq t$, $s,t\in I$, result, which prove that \eqref{varad} has an exponential dichotomy on $I$ with projector $Q_\lambda$. The last statement follows from \cite[p.~294, Prop.~6(ii)]{zeidler:95}.
\end{proof}
The following result was already partly stated in \cite[Prop.~3.1]{poetzsche:12}:
\begin{theorem}[Fredholmness]\label{thmfred}
	If $(H_0$--$H_1)$ hold, then the following are equivalent for all parameters $\lambda\in\tilde\Lambda$:
	\begin{enumerate}
		\item[(a)] $D_1G(0,\lambda)\in L(W^{1,\infty}(\R),L^\infty(\R))$ is Fredholm,

		\item[(b)] $D_1G(0,\lambda)\in L(W_0^{1,\infty}(\R),L_0^\infty(\R))$ is Fredholm,

		\item[(c)] the bounded entire solution $\phi_\lambda:\R\to\Omega$ to \eqref{cde} is hyperbolic on $\R_+$ with Morse index $m_\lambda^+$ and on $\R_-$ with Morse index $m_\lambda^-$,
	\end{enumerate}
	where $\ind D_1G(0,\lambda)=m_\lambda^--m_\lambda^+$.
\end{theorem}
Although the proof literally follows the lines of \cite[Lemma~4.2]{palmer:84}, for further reference in the subsequent text we provide some necessary details.
\begin{proof}
	Throughout, let $\lambda\in\tilde\Lambda$ be fixed.

	$(c)\Rightarrow(a)$ Our assumptions imply that \eqref{var} has exponential dichotomies on $\R_\pm$ with projectors $P_\lambda^\pm$; the corresponding growth rates may be denoted by $\alpha>0$. We establish that $D_1G(0,\lambda):W^{1,\infty}(\R)\to L^\infty(\R)$ is Fredholm.

	(I) \underline{Claim}: \emph{$\dim N(D_1G(0,\lambda))<\infty$.}\\
	Thanks to the dynamical characterization \eqref{R-N}, if we abbreviate
	\begin{align*}
		X_+&:=R(P_\lambda^+(0)),&
		X_-&:=N(P_\lambda^-(0)),
	\end{align*}
	then the intersection $X_+\cap X_-$ is precisely the subspace of all initial values $\xi\in\R^d$ for bounded entire solutions $\Phi_\lambda(\cdot,0)\xi$ to \eqref{var}. By means of the isomorphism $\xi\mapsto\Phi_\lambda(\cdot,0)\xi$ from $\R^d$ onto the solution space of \eqref{var} one has $\dim(X_+\cap X_-)=\dim N(D_1G(\lambda,0))$.

	(II) \underline{Claim}: \emph{If $g\in R(D_1G(0,\lambda))$, then for all solutions $\psi\in L^\infty(\R)$ of the dual variation equation \eqref{varad} one has
	\begin{equation}
		\int_{\R}\sprod{\psi(s),g(s)}\d s=0.
		\label{pal35}
	\end{equation}}
	First of all, \lref{lemad} implies that the dual variation equation \eqref{varad} has exponential dichotomies on $\R_\pm$ with projectors $Q_\lambda^\pm(t)=I_d-P_\lambda^\pm(t)^T$ and growth rate $\alpha>0$. Therefore the following orthogonal complements allow the dynamical characterization
	\begin{align}\label{R-Nad}
		\begin{split}
			X_+^\perp&=R(Q_\lambda^+(0))
			=
			\set{\eta\in\R^d:\,\sup_{0\leq t}e^{-\gamma t}|\Phi_\lambda^\ast(t,0)\eta|<\infty}
			\text{ for }\gamma\in[-\alpha,\alpha),\\
			X_-^\perp&=N(Q_\lambda^-(0))
			=
			\set{\eta\in\R^d:\,\sup_{t\leq 0}e^{-\gamma t}|\Phi_\lambda^\ast(t,0)\eta|<\infty}
			\text{ for }\gamma\in(-\alpha,\alpha].
		\end{split}
	\end{align}
	Hence, the intersection $X_+^\perp\cap X_-^\perp$ consists of initial values $\eta$ for bounded entire solutions $\Phi_\lambda^\ast(\cdot,0)\eta$ to the dual variation equation \eqref{varad}.
	For each $g\in R(D_1G(0,\lambda))$ there exists a preimage $\phi\in W^{1,\infty}(\R)$ such that the identity $g(t)\equiv\dot\phi(t)-D_2f(t,\phi_\lambda(t),\lambda)\phi(t)$ holds a.e.\ on $\R$. If $\psi\in L^\infty(\R)$ denotes a solution of \eqref{varad}, then the product rule implies
	\begin{align*}
		\int_{-T}^T\sprod{\psi(s),g(s)}\d s
		&=
		\int_{-T}^T\sprod{\psi(s),\dot\phi(s)-D_2f(s,\phi_\lambda(s),\lambda)\phi(s)}\d s\\
		&=
		\int_{-T}^T\sprod{\psi(s),\dot\phi(s)}+\sprod{\dot\psi(s)\phi(s)}\d s
		=
		\int_{-T}^T\frac{\d}{\d s}\sprod{\psi(s),\phi(s)}\d s\\
		&=
		\sprod{\psi(T),\phi(T)}-\sprod{\psi(-T),\phi(-T)}
		\fall T>0.
	\end{align*}
	Since the dynamical characterization \eqref{R-Nad} guarantees that $\psi(t)$ decays to $0$ as $t\to\pm\infty$ exponentially and $\phi\in L^\infty(\R)$ holds, one obtains \eqref{pal35} from
	\begin{equation*}
		\int_{\R}\sprod{\psi(s),g(s)}\d s=\lim_{T\to\infty}\intoo{\sprod{\psi(T),\phi(T)}-\sprod{\psi(-T),\phi(-T)}}=0.
	\end{equation*}

	(III) \underline{Claim}: \emph{If \eqref{pal35} holds for all solutions $\psi\in L^\infty(\R)$ of the dual variation equation \eqref{varad}, then $g\in R(D_1G(0,\lambda))$.}\\
	Let $g\in L^\infty(\R)$. For $\eta\in\R^d$ satisfying
	\begin{equation}
		\eta^T\intcc{P_\lambda^+-(I_d-P_\lambda^-)}=0
		\label{pal36}
	\end{equation}
	we define
	\begin{align*}
		\psi:\R&\to\R^{d\tm d},&
		\psi(t)&:=
		\begin{cases}
			\Phi_\lambda^\ast(t,0)Q_\lambda^+(t)\eta,&t\geq 0,\\
			\Phi_\lambda^\ast(t,0)Q_\lambda^-(t)\eta,&t\leq 0.
		\end{cases}
	\end{align*}
	Then $\psi$ is a bounded entire solution of the dual variation equation \eqref{varad} and
	\begin{equation*}
		\sprod{\eta,\int_{-\infty}^0P_\lambda^-(0)\Phi_\lambda(0,s)g(s)\d s+\int_0^\infty(I_d-P_\lambda^+(0))\Phi_\lambda(0,s)g(s)\d s}=0
	\end{equation*}
	for all $\eta\in\R^d$ such that \eqref{pal36} holds. This, in turn, is equivalent to the fact that the linear algebraic equation
	\begin{multline*}
		\intcc{P_\lambda^+(0)-(I_d-P_\lambda^-(0))}\xi\\
		=
		\int_{-\infty}^0P_\lambda^-(0)\Phi_\lambda(0,s)g(s)\d s+\int_0^\infty(I_d-P_\lambda^+(0))\Phi_\lambda(0,s)g(s)\d s
	\end{multline*}
	for $\xi\in\R^d$ has a solution. Consequently, with Green's function \eqref{nogreen},
	\begin{align*}
		\phi:\R&\to\R^d,&
		\phi(t)
		&:=
		\begin{cases}
			\Phi_\lambda(t,0)P_\lambda^+(0)\xi+\int_0^\infty\Gamma_{P_\lambda^+}(t,s)g(s)\d s,&t\geq 0,\\
			\Phi_\lambda(t,0)[I_d-P_\lambda^-(0)]\xi+\int_{-\infty}^0\Gamma_{P_\lambda^-}(t,s)g(s)\d s,&t\leq 0
		\end{cases}
	\end{align*}
	defines a solution of the perturbed variation equation \eqref{linb} in $W^{1,\infty}(\R)$. Due to \eqref{thmgprop1} this means $D_1G(0,\lambda)\phi=g$, i.e.\ $g\in R(D_1G(0,\lambda))$.

	(IV) \underline{Claim}: \emph{$\codim R(D_1G(0,\lambda))<\infty$.}\\
	For each bounded entire solution $\psi$ of \eqref{varad} we define a functional
	\begin{align*}
		x_\psi'&\in W^{1,\infty}(\R)',&
		x_\psi'(x)&:=\int_{\R}\sprod{\psi(s),x(s)}\d s.
	\end{align*}
	This establishes an isomorphism between $X_+^\perp\cap X_-^\perp$ and a finite-dimensional subspace of the dual space $L^\infty(\R)'$. In other words, $R(D_1G(0,\lambda))$ is the subspace of $L^\infty(\R)$ annihilated by this finite-dimensional subspace of $L^\infty(\R)'$. Consequently, $R(D_1G(0,\lambda))$ is closed, $\codim R(D_1G(0,\lambda))=\dim(X_+^\perp\cap X_-^\perp)<\infty$ and $D_1G(0,\lambda)$ is Fredholm.

	(V) It remains to determine the index of $D_1G(0,\lambda)$ as
	\begin{align*}
		&
		\ind D_1G(0,\lambda)\\
		&=
		\dim\bigl(R(P_\lambda^+(0))\cap N(P_\lambda^-(0))\bigr)-\dim\bigl(N(P_\lambda^-(0))+R(P_\lambda^+(0))\bigr)^{\perp}\\
		&=
		\dim\intoo{R(P_\lambda^+(0))\cap N(P_\lambda^-(0))}-\left(d-\dim\bigl(N(P_\lambda^-(0))+R(P_\lambda^+(0))\bigr)\right)\\
		&=
		\dim\intoo{R(P_\lambda^+(0))\cap N(P_\lambda^-(0))}\\
		&\quad\quad\quad-\left(d-[\dim N(P_\lambda^-(0))+\dim R(P_\lambda^+(0))-\dim R(P_\lambda^+(0))\cap N(P_\lambda^-(0)) ]\right)\\
		&=\dim R(P_\lambda^+(0))-\bigl(d-\dim N(P_\lambda^-(0))\bigr)
		=
		m_\lambda^--m_\lambda^+.
	\end{align*}

	$(a)\Rightarrow(b)$ In order to show that $D_1G(0,\lambda)\in L(W_0^{1,\infty}(\R),L_0^\infty(R))$ is Fredholm, we mimic the arguments in (a). This additionally only requires to establish the inclusions $\phi,\psi\in W_0^{1,\infty}(\R)$, provided that $g\in L_0^\infty(\R)$ holds, but the corresponding estimates result as in the proof of \tref{thmadmin}(b).

	$(b)\Rightarrow(c)$ Referring to \lref{lemaprop}(a), the coefficient matrices $A(\cdot,\lambda):\R\to\R^{d\tm d}$ are essentially bounded. Then the proof of \cite[Thm.]{palmer:88} given for continuous $A(\cdot,\lambda)$ and the spaces $(BC^1(\R),BC(\R))$ literally carries over to our situation.
\end{proof}
\section{Evans function and parity}
\label{sec3}
This section contains our main result. It relates two seemingly independent concepts, namely the Evans function of a parametrized family of variation equations \eqref{var} to the parity of an abstract path of index $0$ Fredholm operators (cf.~App.~\ref{appA}). As basis for the Fredholm properties we impose
\begin{hypothesis*}[$\mathbf{H_2}$]
	There exists a \emph{critical parameter} $\lambda^\ast\in\tilde\Lambda$ such that the entire solution $\phi^\ast:=\phi_{\lambda^\ast}$ of $(C_{\lambda^\ast})$ is hyperbolic on both semiaxes $\R_+$ with projector $P_{\lambda^\ast}^+:\R_+\to\R^{d\tm d}$ (Morse index $m^+$) and on $\R_-$ with projector $P_{\lambda^\ast}^-:\R_-\to\R^{d\tm d}$ (Morse index $m^-$).
\end{hypothesis*}
This local assumption extends to a neighborhood of $\lambda^\ast$ as follows:
\begin{lemma}[roughness]\label{lemrough}
	Let $(H_0$--$H_2)$ hold, where the exponential dichotomies on the semiaxes have the growth rate $\alpha^\ast>0$. If $\alpha\in(0,\alpha^\ast)$, then there exists a $\rho_0>0$ such that for each $\lambda\in\bar B_{\rho_0}(\lambda^\ast)$ the solution $\phi_\lambda$ is hyperbolic on both $\R_+$ with a projector $P_\lambda^+$ and on $\R_-$ with a projector $P_\lambda^-$ and common growth rate $\alpha$. Moreover, the projection mappings $(t,\lambda)\mapsto P_\lambda^+(t)$ and $(t,\lambda)\mapsto P_\lambda^-(t)$ are continuous with
	\begin{align*}
		\dim R(P_\lambda^+(0))&\equiv d-m^+,&
		\dim N(P_\lambda^-(0))&\equiv m^-\on B_{\rho_0}(\lambda^\ast).
	\end{align*}
\end{lemma}
The characterization in \tref{thmfred} provides an easy proof that the hyperbolicity of $\phi^\ast$ extends to the bounded entire solutions $\phi_\lambda$ for parameters $\lambda$ near $\lambda^\ast$. Thereto, Hypothesis~$(H_2)$ implies that $D_1G(0,\lambda^\ast)$ is Fredholm. Since the Fredholm operators form an open subset of $L(W^{1,\infty}(\R),L^\infty(\R))$ (cf.~\cite[p.~300, Prop.~1]{zeidler:95}) and $\lambda\mapsto D_1G(0,\lambda)$ is continuous by \tref{thmgprop}, also $D_1G(0,\lambda)$ is Fredholm (preserving the index). This, in turn, ensures that $\phi_\lambda$ is hyperbolic on both semiaxes for $\lambda$ in a neighborhood of $\lambda^\ast$.
\begin{proof}
	As in \cite[S.~8, Lemma~1.1]{sandstede:93} the projectors $P_\lambda^\pm$ are characterized via fixed points of Lyapunov--Perron operators. Then their continuous dependence on the parameter $\lambda$ is a consequence of the Uniform Contraction Principle.
\end{proof}

Under Hypothesis~$(H_2)$ we subsequently fix a growth rate $\alpha\in(0,\alpha^\ast)$ and based on $\rho_0>0$ from \lref{lemrough} consider the parameter space
\begin{equation*}
	\Lambda:=\set{\lambda\in\tilde\Lambda:\,d(\lambda,\lambda^\ast)\leq\rho_0}.
\end{equation*}
\begin{lemma}\label{bundle2}
	Under assumptions $(H_0$--$H_2)$ the sets
	\begin{align*}
		\mathrm{R}[P^{+}(t)]&:=\set{(\lambda,\xi)\in\Lambda\tm\R^d\mid \xi\in R(P_\lambda^+(t))}\fall t\in\R_+,\\
		\mathrm{N}[P^{-}(t)]&:=\set{(\lambda,\xi)\in\Lambda\tm\R^d\mid \xi\in N(P_\lambda^-(t))}\fall t\in\R_-
	\end{align*}
	are vector bundles of dimension $d-m^+$ resp.\ $m^-$ over $\Lambda$.
\end{lemma}
\begin{proof}
	Since $\lambda\mapsto P_\lambda^{\pm}(t)$ is continuous for any $t\in\R_\pm$ by \lref{lemrough}, the assertion follows directly from \cite[Prop.~6.21]{FiPejsachowiczIII}.
\end{proof}
\begin{proposition}[Evans function]\label{propetagamma}
	Let $(H_0$--$H_2)$ hold with Morse indices $m^+=m^-$. If $\Lambda$ is contractible, then there exist continuous functions $\xi^+_1,\ldots,\xi_{d-m^+}^+:\Lambda\to\R^d$ and $\xi^-_1,\ldots,\xi_{m^-}^-:\Lambda\to\R^d$ such that
	\begin{enumerate}
		\item[$(a)$] $\xi^+_1(\lambda),\ldots,\xi_{d-m^+}^+(\lambda)$ is a base of $R(P_\lambda^+(0))\subseteq\R^d$,

		\item[$(b)$] $\xi^-_1(\lambda),\ldots,\xi_{m^-}^-(\lambda)$ is a base of $N(P_\lambda^-(0))\subseteq\R^d$,

		\item[$(c)$] $\xi^+_1(\lambda),\ldots,\xi_{d-m^+}^+(\lambda),\xi^-_1(\lambda),\ldots,\xi_{m^-}^-(\lambda)$ is a basis of $\R^d$ if and only if one has the direct sum $R(P_\lambda^+(0))\oplus N(P_\lambda^-(0))=\R^d$
	\end{enumerate}
	holds for all $\lambda\in\Lambda$. Given this, an \emph{Evans function} for \eqref{var} is defined by
	\begin{align*}
		E:\Lambda&\to\R,&
		E(\lambda)&:=\det\bigl(\xi^+_1(\lambda),\ldots,\xi_{d-m^+}^+(\lambda),\xi^-_1(\lambda),\ldots,\xi_{m^-}^-(\lambda)\bigr).
	\end{align*}
\end{proposition}

Evans functions clearly depend on the choice of $\xi^+_i(\lambda)$ and $\xi^-_j(\lambda)\in\R^d$. However, any two Evans functions differ only by a product with a nonvanishing function (this factor is a determinant of the transformation matrices that describe the change of bases). Furthermore, all subsequent statements involving zeros of the Evans function do not depend on the choice of the basis vectors.
\begin{proof}
	We write $\R^d=(\R^{d-m^+}\tm\{0\})\oplus(\{0\}\tm\R^{m^-})$ and consider the sets $\mathrm{R}[P^{+}(0)]$ and $\mathrm{N}[P^{-}(0)]$ introduced in \lref{bundle2}. Since they are are vector bundles over a contractible space, \cite[p.~30, Cor.~4.8]{husemoller:74} yields the existence of morphism bundles
\begin{align*}
	\psi^1:\Lambda\tm\R^{d-m^+}&\to\mathrm{R}[P^{+}(0)],&
	\psi^2:\Lambda\tm\R^{m^-}&\to\mathrm{N}[P^{-}(0)]
\end{align*}
such that the following diagrams are commutative:
\begin{align*}
\xymatrix{
\Lambda\tm \R^{d-m^+}\ar[r]^-{\psi^1}_{\cong}\ar[d]_{}& \mathrm{R}[P^{+}(0)]\ar[d]_{}\\
\Lambda\ar[r]^-{\id}&\Lambda,
}& &
\xymatrix{
\Lambda\tm \R^{m^-}\ar[r]^-{\psi^2}_{\cong}\ar[d]_{}& \mathrm{N}[P^{-}(0)]\ar[d]_{}\\
\Lambda\ar[r]^-{\id}&\Lambda,
}
\end{align*}
	where the vertical arrows represent the corresponding projections, and for any $\lambda\in\Lambda$ the maps $\psi^1(\lambda,\cdot):\R^{d-m^+}\to R(P_\lambda^+(0))$ and $\psi^2(\lambda,\cdot):\R^{m^-}\to N(P_\lambda^-(0))$ are linear isomorphisms. Then it is not hard to see that the functions
	\begin{align*}
		\xi^+_1,\ldots,\xi_{d-m^+}^+:\Lambda&\to\R^d,&
		\xi^+_i(\lambda)&:=\psi^1(\lambda,e_i)\fall 1\leq i\leq d-m^+,\\
		\xi^-_1,\ldots,\xi_{m^-}^-:\Lambda&\to \R^d,&
		\xi^-_j(\lambda)&:=\psi^2(\lambda,e_{d-m^++j})\fall 1\leq j\leq m^-
	\end{align*}
	satisfy the claimed properties $(a$--$c)$ with the sets $\{e_1,\ldots,e_{d-m^+}\}\subset\R^{d-m^+}\times\{0\}$ and $\set{e_{d-m^++1},\ldots,e_d}\subset\set{0}\tm\R^{m^-}$ derived from the standard basis $e_1,\ldots,e_d\in\R^d$. In particular, $E:\Lambda\to\R$ is well-defined.
\end{proof}

The following observations recommend Evans functions $\lambda\mapsto E(\lambda)$ as tool to locate non-trivial intersections of the stable and unstable vector bundles to \eqref{var} (transversality). In addition, it extends the admissibility \tref{thmadmin}:
\begin{proposition}[properties of Evans functions]\label{propEvans}
	Let $(H_0$--$H_2)$ hold with Morse indices $m^+=m^-$. If $\Lambda$ is contractible, then an Evans function $E:\Lambda\to\R$ of \eqref{var} is continuous. Moreover, for each $\lambda\in\Lambda$ the following are equivalent:
	\begin{enumerate}
		\item[$(a)$] $E(\lambda)\neq 0$,

		\item[$(b)$] $R(P_\lambda^+(0))\oplus N(P_\lambda^-(0))=\R^d$,

		\item[$(c)$] a bounded entire solution $\phi_\lambda$ of \eqref{cde} is hyperbolic on $\R$, i.e.\ $0\not\in\Sigma(\lambda)$.
	\end{enumerate}
\end{proposition}
\begin{proof}
	The continuity of the Evans function follows immediately from \pref{propetagamma} and the continuity of the determinant $\det:\R^{d\tm d}\to\R$.

	$(a)\Leftrightarrow(b)$ is an immediate consequence of Linear Algebra.

	$(b)\Leftrightarrow(c)$ Arguing as in \cite[p.~19]{coppel:78}, a variation equation \eqref{var} possesses an exponential dichotomy on $\R$ if and only if the projections satisfy (b).
\end{proof}

An Evans function serves as indicator when spectral intervals split (see Fig.~\ref{figspectrum}):
\begin{SCfigure}[2]
	\includegraphics[scale=0.5]{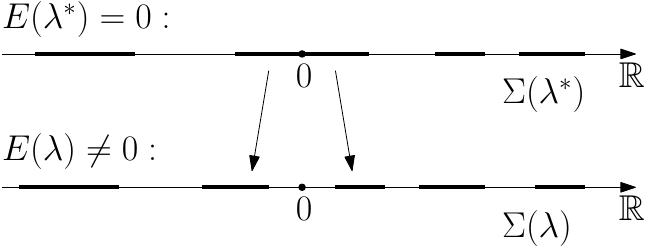}
	\caption{To \cref{corEvans}: Dichotomy spectra $\Sigma(\lambda)$ of \eqref{var}: At zeros $\lambda^\ast$ of an Evans function $E$ the critical spectral interval of $\Sigma(\lambda^\ast)$ is not a singleton. For $E(\lambda)\neq 0$ the interval splits, which results in $0\not\in\Sigma(\lambda)$ and hyperbolic solutions $\phi_\lambda$ to \eqref{cde}}
	\label{figspectrum}
\end{SCfigure}

\begin{corollary}[Evans function and dichotomy spectrum]
	\label{corEvans}
	The following are equivalent:
	\begin{enumerate}
		\item[$(a)$] $E(\lambda^\ast)=0$,

		\item[$(b)$] there exists a $\alpha^\ast>0$ such that $[-\alpha^\ast,\alpha^\ast]\subseteq\Sigma(\lambda^\ast)$.
	\end{enumerate}
	Moreover, if (a) (or (b)) holds for some $\lambda^\ast\in\Lambda$, then for all $\lambda\in\Lambda$ one has
	\begin{enumerate}
		\item[$(c)$] $0<m^+,m^-<d$ and each solution $\phi_\lambda$ of \eqref{cde} is unstable,
		
		\item[$(d)$] in case $\lambda\in\Lambda\setminus E^{-1}(0)$ it is $0\not\in\Sigma(\lambda)$ such that both $(-\infty,0)$ and $(0,\infty)$ contain at least one spectral interval of $\Sigma(\lambda)$.
	\end{enumerate}
\end{corollary}
\begin{proof}
	(I) We first establish the claimed equivalence:

	$(a)\Rightarrow(b)$ Hypothesis~$(H_2)$ implies that $(V_{\lambda^\ast})$ has exponential dichotomies on both semiaxes with growth rate $\alpha^\ast>0$. Now if $E(\lambda^\ast)=0$, then \pref{propEvans} shows that $\phi^\ast$ is nonhyperbolic, i.e.\ $0\in\Sigma(\lambda^\ast)$. Due to \tref{thmadmin}, $D_1G(0,\lambda^\ast)$ is noninvertible, but because of \tref{thmfred} Fredholm of index $0$. Hence, $D_1G(0,\lambda^\ast)$ has a nontrivial kernel and from step (I) in the proof for \tref{thmfred} one obtains $\set{0}\neq R(P_{\lambda^\ast}^+(0))\cap N(P_{\lambda^\ast}^-(0))$. Because the transition matrices of \eqref{var} and $\Phi_\lambda^\gamma$ of $\dot x=[A(t,\lambda)-\gamma I_d]x$ are related by $\Phi_\lambda^\gamma(t,s)=e^{\gamma(s-t)}\Phi_\lambda(t,s)$ for all $t,s\in\R$, the dynamical characterization \eqref{R-N} implies
	\begin{equation*}
		\set{0}
		\neq
		\set{\xi\in\R^d:\,\sup_{t\in\R}e^{-\gamma t}\abs{\Phi_{\lambda^\ast}(t,0)\xi}<\infty}
		=
		\set{\xi\in\R^d:\,\sup_{t\in\R}\abs{\Phi_{\lambda^\ast}^\gamma(t,0)\xi}<\infty}
	\end{equation*}
	for each $\gamma\in(-\alpha^\ast,\alpha^\ast)$. Consequently, the shifted equation $\dot x=[A(t,\lambda^\ast)-\gamma I_d]x$ has nontrivial bounded solutions and thus cannot possess an exponential dichotomy on $\R$, i.e.\ $\gamma\in\Sigma(\lambda^\ast)$. Since $\gamma\in(-\alpha^\ast,\alpha^\ast)$ was arbitrary, we deduce $(-\alpha^\ast,\alpha^\ast)\subseteq\Sigma(\lambda^\ast)$ and (b) results due to the compactness of spectral intervals.

	$(b)\Rightarrow(a)$ Because obviously $0\in\Sigma(\lambda^\ast)$ holds, the bounded entire solution $\phi^\ast$ is nonhyperbolic and \pref{propEvans} yields (a).

	(II) Let $\lambda\in\Lambda$. By means of contradiction we assume $P_{\lambda^\ast}^+(0)=I_d$ and hence the vectors $\xi^+_1(\lambda),\ldots,\xi_d^+(\lambda)\in\R^d$ from \pref{propetagamma} are linearly independent. Therefore, $E(\lambda)\neq 0$, contradicting $E(\lambda^\ast)=0$ holds, and we deduce $\dim R(P_{\lambda^\ast}^+(0))<d$, whence $m^+=d-\dim R(P_{\lambda^\ast}^+(0))>0$. Using a similar argument, also the assumption $P_{\lambda^\ast}^+(0)=0_d$ yields a contradiction, which yields $m^+<d$. With $m^-=m^+$ it is $0<m^+,m^-<d$.\\
	This implies that the projectors $P_\lambda^+$ for the exponential dichotomy of \eqref{var} on $\R_+$ are nontrivial. Then \cite[pp.~32--33, Prop.~2.2.1]{anagnostopoulou:poetzsche:rasmussen:22} implies that $\phi_\lambda$ is unstable.\\
	Now let $E(\lambda)\neq 0$ and hence $0\not\in\Sigma(\lambda)$ results by \pref{propEvans}. First, the assumption $\Sigma(\lambda)\subset(-\infty,0)$ implies that \eqref{var} has an exponential dichotomy on $\R_+$ with $P_\lambda(t)\equiv I_d$ resulting in the contradiction $m^+=0$. Second, assuming $\Sigma(\lambda)\subset(0,\infty)$ leads to the contradiction $m^+=d$. In conclusion, both the positive and the negative reals contains at least one spectral interval of $\Sigma(\lambda)$.
\end{proof}

As an interim conclusion, note that Evans functions have a clear geometric interpretation and are accessible in practice (cf.~\cite{dieci:elia:vanvleck:11}). Furthermore, referring to Thms.~\ref{thmadmin} and \ref{thmfred}, as well as \pref{propEvans}, they allow to distinguish invertibility from merely Fredholmness of the partial derivatives $D_1G(0,\lambda)$ given in \eqref{thmgprop1}. We thus aim to relate Evans functions to abstract bifurcation theory. Here the concept of parity, explained in App.~\ref{appA} and \ref{appB}, is crucial. In our present framework, we further restrict to interval neighborhoods
\begin{equation*}
	\Lambda:=[a,b]\quad\text{with reals }a<b,\,a,b\in \bar B_{\rho_0}(\lambda^\ast)
\end{equation*}
and arrive at our central result involving the parity $\sigma$:
\begin{theorem}[Evans function and parity]\label{thmmain}
	Let $(H_0$--$H_2)$ hold with Morse indices $m^+=m^-$.
	If $E(a)\cdot E(b)\neq 0$, then both mappings
	\begin{enumerate}
		\item[$(a)$] $T\colon [a,b]\to L(W^{1,\infty}(\R),L^{\infty}(\R))$,

		\item[$(b)$] $T\colon [a,b]\to L(W_0^{1,\infty}(\R),L_0^{\infty}(\R))$,
	\end{enumerate}
	given by $T(\lambda):=D_1G(0,\lambda)$ define paths of index $0$ Fredholm operators with invertible endpoints and parity $\sigma(T,[a,b])=\sgn E(a)\cdot\sgn E(b)$.
\end{theorem}
\begin{proof}
	(a) The mapping $T:[a,b]\to L(W^{1,\infty}(\R),L^{\infty}(\R))$ is continuous because of \tref{thmgprop}(a). Due to \lref{lemrough} the variation equations \eqref{var} have exponential dichotomies on both semiaxes $\R_\pm$ with respective projectors $P_\lambda^\pm$ for all $\lambda\in\Lambda$. Thus, the inclusion $T(\lambda)\in F_0(W^{1,\infty}(\R),L^{\infty}(\R))$ results from \tref{thmfred}(a) combined with \eqref{thmgprop}. Hence, $T$ is a path of index $0$ Fredholm operators. Furthermore, $E(a)E(b)\neq 0$ and \pref{propEvans} yield that both bounded entire solutions $\phi_a$ and $\phi_b$ are hyperbolic, i.e.\ the variation equations $(V_a)$ and $(V_b)$ have an exponential dichotomy on $\R$. Then \tref{thmadmin}(a) implies the inclusions $T(a),T(b)\in GL(W^{1,\infty}(\R),L^{\infty}(\R))$ and $T$ has invertible endpoints.

	We first prepare some properties of $T(\lambda)$ needed in the further steps of the proof. Throughout, we again abbreviate $A(t,\lambda):=D_2f(t,\phi_\lambda(t),\lambda)$.
	
	(I)~\underline{Claim}: \emph{$N(T(\lambda))\cong R(P_\lambda^+(0))\cap N(P_\lambda^-(0))$.}\\
	This is shown in the proof of \tref{thmfred}(a).

	(II) Consider the formally dual operators
	\begin{align*}
		T(\lambda)^{\ast}&:W^{1,\infty}(\R)\to L^{\infty}(\R),&
		[T(\lambda)^{\ast}y](t)&:=\dot{y}(t)+A(t,\lambda)^Ty(t)
	\end{align*}
	for a.a.\ $t\in\R$ associated to the dual variation equation \eqref{varad}. Thanks to \lref{lemad}, the dual variation equation \eqref{varad} has exponential dichotomies on both $\R_\pm$ with the projectors $Q_{\lambda}^{\pm}(t):=I_d-P_{\lambda}^{\pm}(t)^T$.

	(II.1)~\underline{Claim}: $\dim N(T(\lambda))=\dim N(T(\lambda)^\ast)$.\\
	As in the proof of \tref{thmfred}(a) (cf.~\eqref{R-Nad}) one has
	\begin{equation*}
		N(T(\lambda)^{\ast})
		\cong
		R(Q^+_{\lambda}(0))\cap	N(Q_\lambda^-(0))
		\stackrel{\eqref{proad}}{=}
		R(P_\lambda^-(0)^T)\cap N(P_\lambda^+(0)^T),
	\end{equation*}
	while $R(P_\lambda^-(0)^T)=N(P_\lambda^-(0))^\perp$, $N(P_\lambda^+(0)^T)=R(P_\lambda^+(0))^\perp$ result from \cite[p.~294, Prop.~6(ii)]{zeidler:95} and consequently the claim is established by
	\begin{align*}
		&
		\dim N(T(\lambda)^\ast)
		=
		\dim\intoo{R(P_\lambda^-(0)^T)\cap N(P_\lambda^+(0)^T)}\\
		&=
		\dim\intoo{N(P_\lambda^-(0))^\perp\cap R(P_\lambda^+(0))^\perp}
		=
		\dim\intoo{N(P_\lambda^-(0))+R(P_\lambda^+(0))}^\perp\\
		&=
		\dim\bigl(R(P_\lambda^+(0))\cap N(P_\lambda^-(0))\bigr)
		\stackrel{(I)}{=}
		\dim N(T(\lambda)).
	\end{align*}

	(II.2) \underline{Claim}: $R(T(\lambda))\oplus N(T(\lambda)^{\ast})=L^{\infty}(\R)$.\\
	First of all, observe that since $W^{1,\infty}(\R)\subset L^{\infty}(\R)$, it follows that $N(T(\lambda)^\ast)$ is a subset of $L^{\infty}(\R)$.
Furthermore, if $g\in R(T(\lambda))$ with preimage $\phi\in W^{1,\infty}(\R)$ and functions $u\in N(T(\lambda)^{\ast})\subseteq W^{1,\infty}(\R)$, then
	\begin{align*}
		&\int_{\R}\sprod{u(s),g(s)}\d s
		=
		\int_{\R}\sprod{u(s),\dot{\phi}(s)-A(t,\lambda)\phi(s)}\d s\\
		&=
		\int_{\R}\sprod{u(s),\dot{\phi}(s)}-\sprod{u(s),A(t,\lambda)\phi(s)}\d s\\
		&=\int_{\R}\sprod{u(s),\dot{\phi}(s)}+\sprod{-u(s)A(t,\lambda)^T,\phi(s)}\d s\\
		&=\int_{\R}\sprod{u(s),\dot{\phi}(s)}+\sprod{\dot{u}(s),\phi(s)}\d s
		=
		\int_{\R}\frac{\d\sprod{u(s),\phi(s)}}{\d s}\d s
	\end{align*}
	due to the product rule. Now, reasoning as in the proof \tref{thmfred}, one obtains
	\begin{align*}
		&
		\int_{\R}\sprod{u(s),g(s)}\d s
		=
		\int_{-\infty}^0\frac{\d\sprod{u(s),\phi(s)}}{\d s}\d s+
		\int_0^{\infty}\frac{\d\sprod{u(s),\phi(s)}}{\d s}\d s\\
		&=
		\sprod{u(0),\phi(0)}-\lim_{t\to-\infty}\sprod{u(t),\phi(t)}+
		\lim_{t\to\infty}\sprod{u(t),\phi(t)}-\sprod{u(0),\phi(0)}
		=
		0.
	\end{align*}
	In particular, from this we obtain $R(T(\lambda))\cap N(T(\lambda)^{\ast})=\set{0}$, and combined with $\ind T(\lambda)=0$ and (II.1) we arrive at the desired splitting.

	(III) We construct a finite-dimensional subspace $V\subseteq L^{\infty}(\R)$ complementary to each $R(T(\lambda))$, i.e., $R(T(\lambda))+V=L^{\infty}(\R)$ holds for $\lambda\in [a,b]$. Thereto, if $\lambda_0\in[a,b]$ is fixed, then because of $\dim N(T(\lambda_0))<\infty$ there exists a complement $W_{\lambda_0}\subset W^{1,\infty}(\R)$ with $N(T(\lambda_0))\oplus W_{\lambda_0}=W^{1,\infty}(\R)$. Now consider the bilinear operator
	\begin{align*}
		\Sigma_{\lambda}: W_{\lambda_0}\tm N(T(\lambda_0)^\ast)&\to L^{\infty}(\R),&
		\Sigma_{\lambda}(w,v):=T(\lambda)w+v.
	\end{align*}
	Because of $\Sigma_{\lambda_0}\in GL(W_{\lambda_0}\tm N(T(\lambda_0)^\ast),L^{\infty}(\R))$ it follows from the continuity of $T$ and the openness of the set of bounded invertible operators that there exists a neighborhood $U_{\lambda_0}$ of $\lambda_0$ in $[a,b]$ such that $\Sigma_{\lambda}\in GL(W_{\lambda_0}\tm N(T(\lambda_0)^\ast),L^{\infty}(\R))$ and hence
	\begin{equation*}
		R(T(\lambda))+N(T(\lambda_0)^\ast)= L^{\infty}(\R)\fall\lambda\in U_{\lambda_0}.
	\end{equation*}
	By compactness we can now cover the interval $[a,b]$ with a finite number of such neighborhoods $U_{\lambda_1},\ldots,U_{\lambda_n}\subseteq\R$. If
\begin{align}\label{transversal-V}
	V:=N(T(\lambda_1)^\ast)+N(T(\lambda_2)^\ast)+\ldots+N(T(\lambda_n)^\ast)\subset L^\infty(\R),
\end{align}
then $\dim V<\infty$ and $R(T(\lambda))+ V= L^{\infty}(\R)$ for all $\lambda\in [a,b]$.

	(IV) This step is inspired by Step~$3$ in the proof of \cite[Thm.~5.3]{NilsLag}. Keeping $\tau>0$ fixed, consider the family of operators
	\begin{align*}
		S(\lambda):D(S(\lambda))&\to L^{\infty}[-\tau,\tau],&
		[S(\lambda)y](t)&:=\dot{y}(t)-A(t,\lambda)y(t)
	\end{align*}
	for a.a.\ $t\in[-\tau,\tau]$, which due to \lref{lemaprop}(a) is well-defined on the domain
	\begin{equation*}
		D(S(\lambda))
		:=
		\set{u\in W^{1,\infty}[-\tau,\tau]\mid u(-\tau)\in N(P_\lambda^-(-\tau)),\, u(\tau)\in R(P_\lambda^+(\tau))}.
	\end{equation*}

	(IV.1) \underline{Claim}: $\dim N(S(\lambda))=\dim N(T(\lambda))<\infty$.\\
	Consider the commutative diagram
	\begin{align}\label{diagram-i}
		\begin{split}
		\xymatrix{
		W^{1,\infty}(\R)\ar[r]^-{T(\lambda)}& L^{\infty}(\R)\ar[d]^p\\
		D(S(\lambda))\ar[r]^-{S(\lambda)}\ar[u]_{i_{\lambda}}& L^{\infty}[-\tau,\tau],}
		\end{split}
	\end{align}
	where $p$ is the restriction of functions in $L^{\infty}(\R)$ to $L^{\infty}[-\tau,\tau]$ given by $p(u):=u|_{[-\tau,\tau]}$ and a canonical map $i_\lambda:D(S(\lambda))\to W^{1,\infty}(\R)$ defined by extending a given function $u\in D(S(\lambda))$ to the intervals $(-\infty,-\tau)$ and $(\tau,\infty)$ as solution of \eqref{var}. Observe that $i_\lambda$ is injective and
	$
		i_{\lambda}\bigl(N(S(\lambda))\bigr)=N(T(\lambda))
	$
	holds, where $i_{\lambda}\bigl(N(S(\lambda))\bigr)\subseteq N(T(\lambda))$ results directly due to the construction of $i_{\lambda}$, while $i_{\lambda}\bigl(N(S(\lambda))\bigr)\supseteq N(T(\lambda))$ as converse inclusion follows from the fact that privided $u\in N(T(\lambda))$, then $u(\tau)\in R(P_\lambda^+(\tau))$ and $u(-\tau)\in N(P_\lambda^-(-\tau))$ for any $\tau>0$ (recall \eqref{ed0}). Finally, this yields the assertion.

	(IV.2) We decompose $[-\tau,\tau]=[-\tau,0]\cup [0,\tau]$ and define the spaces
	\begin{align*}
		X_+&:=\{u\in W^{1,\infty}[0,\tau]\mid u(\tau)\in R(P_\lambda^+(\tau))\,\},&
		Y_+&:=L^{\infty}[0,\tau],\\
		X_-&:=\{u\in W^{1,\infty}[-\tau,0]\mid u(-\tau)\in N(P_\lambda^-(-\tau))\,\},&
		Y_-&:=L^{\infty}[-\tau,0].
	\end{align*}
	Consider the following linear operators $S^\pm(\lambda):X_{\pm}\to Y_{\pm}$ pointwise defined as
	\begin{equation*}
		[S^\pm(\lambda)y](t)=\dot y(t)-D_2f(t,\phi_\lambda(t),\lambda)y(t)\quad\text{for a.e.\ }t\in\R_\pm.
	\end{equation*}
	Next consider the following commutative diagram
	\begin{align}\label{diagram-L}
	\begin{split}
	\xymatrix@C=55pt@R=15pt{
	X_-\oplus X_+\ar[r]^-{S^-(\lambda)\oplus S^+(\lambda)} & Y_-\oplus Y_+\\
	&\\
	D(S(\lambda)) \ar[r]^-{S(\lambda)} \ar[uu]^{J_{\lambda}} & \ar[uu]_{J}
	L^{\infty}[-\tau,\tau],
	}
	\end{split}
	\end{align}
	where $J\colon L^{\infty}[-\tau,\tau]\to Y_-\oplus Y_+$ and $J_{\lambda}\colon D(S(\lambda))\to X_-\oplus X_+$ are defined by
	\begin{equation*}
		Ju:=(u_-,u_+) \text{ and } J_{\lambda}v:=(v_-,v_+)
	\end{equation*}
	with $u_+$ and $v_+$ (resp.\ $u_-$ and $v_-$) being the corresponding restrictions to $[0,\tau]$ (resp.\ to the interval $[-\tau,0])$. It is clear that $J$ is an isomorphism, while $J_{\lambda}$ is injective with range $R(J_{\lambda})=\set{(v_-,v_+)\mid v_-(0)=v_+(0)}$. Defining the mapping $\Sigma\colon X_-\oplus X_+\to \R^d$ by $\Sigma(v_-,v_+)=v_-(0)-v_+(0)$ for $v_{\pm}\in X_{\pm}$, one obtains that $R(J_{\lambda})=N(\Sigma)$. What is more, since $\Sigma$ is an epimorphism, one can conclude that
\begin{align*}
\coker J_{\lambda}=X_-\oplus X_+/R(J_{\lambda})=X_-\oplus X_+/N(\Sigma)\cong \R^d
\end{align*}
and hence $J_{\lambda}$ is Fredholm with $\ind J_{\lambda}=\dim N(J_{\lambda})-\dim \coker J_{\lambda}=0-d=-d$. Since $J$ is an isomorphism, it follows $\ind J^{-1}=0$.

	(IV.3) \underline{Claim}: \emph{$S^+(\lambda)\colon X_+\to Y_+$ and $S^-(\lambda)\colon X_-\to Y_-$ are Fredholm with index $d-m^+$ resp.\ $m^-$}.\\
	Concerning $S^+(\lambda)$ it suffices to establish that
	\begin{equation*}
		S^+(\lambda) \text{ is surjective and }N(S^+(\lambda))\cong R(P_\lambda^+(\tau)),
	\end{equation*}
	which results from the following arguments: Consider the perturbed variation equation \eqref{linb} with inhomogeneity $g\in L^{\infty}[0,\tau]$. Due to \cite[Thm.~2.10]{aulbach:wanner:96} the general solution $\bar\varphi_\lambda$ of \eqref{linb} can be expressed via the Variation of Constants as
	\begin{equation*}
		\bar\varphi_\lambda(t;\tau,\xi)
		=
		\Phi_\lambda(t,\tau)\xi+\int_{\tau}^t\Phi_\lambda(t,s)g(s)\d s\fall t\in [0,\tau],\,\xi\in\R^d.
	\end{equation*}
	Since $\Phi_\lambda$ is a bounded function on $[0,\tau]\times [0,\tau]$ and $g\in L^{\infty}[0,\tau]$ holds, standard calculations yield $\bar\varphi_\lambda(\cdot;\tau,\xi)\in W^{1,\infty}[0,\tau]$, which proves that $S^+(\lambda)$ is onto. Concerning the kernel $N(S^+(\lambda))$, note that $\Phi_\lambda(\cdot,\tau)\xi\in N(S^+(\lambda))$, we conclude $N(S^+(\lambda))\cong R(P_\lambda^+(\tau))$ and finally arrive at $\ind S^+(\lambda)=\dim R(P_\lambda^+(\tau))=d-m^+$ because the invariance relation \eqref{ed0} implies $\dim R(P_\lambda^+(\tau))=\dim R(P_\lambda^+(0))=d-m^+$.

	The argument for $S^-(\lambda)$ is dual to the above proof. Now it suffices to prove that
	\begin{equation*}
		S^-(\lambda)\text{ is surjective and }N(S^-(\lambda))\cong N(P_\lambda^-(-\tau)).
	\end{equation*}
	Thereto, let $g\in L^{\infty}[-\tau,0]$ and note
	\begin{equation*}
		\bar\varphi_\lambda(t;-\tau,\xi)=\Phi_\lambda(t,-\tau)\xi+\int_{-\tau}^t\Phi_\lambda(t,s)g(s)\d s
		\fall t\in[-\tau,0],\,\xi\in\R^d.
	\end{equation*}
	From the boundedness of $\Phi_\lambda$ on the square $[-\tau,0]\times [-\tau,0]$ and $g\in L^{\infty}[-\tau,0]$ results that $\bar\varphi_\lambda(\cdot;-\tau,\xi)\in W^{1,\infty}[-\tau,0]$, which shows that $S^-(\lambda)$ is surjective. Addressing the kernel $N(S^-(\lambda))$ we have $\Phi_\lambda(\cdot,-\tau)\xi\in N(S^-(\lambda))$ and thus $N(S^-(\lambda))\cong N(P_\lambda^-(-\tau))$, leading to $\ind S^-(\lambda)=\dim N(P_\lambda^-(-\tau))=m^-$ by \eqref{ed0}.

	(IV.4) \underline{Claim}: \emph{$S(\lambda):D(S(\lambda))\to L^{\infty}[-\tau,\tau]$ is Fredholm of index $0$.}\\
	Due to the composition $S(\lambda)=J^{-1}\circ (S^-(\lambda)\oplus S^+(\lambda))\circ J_{\lambda}\colon D(S(\lambda)) \to L^{\infty}[-\tau,\tau]$ the commutativity of the diagram \eqref{diagram-L} shows that $S(\lambda)$ is Fredholm with
	\begin{eqnarray*}
		\ind S(\lambda)
		& = &
		\ind(J^{-1}\circ (S^-(\lambda)\oplus S^+(\lambda))\circ J_{\lambda})\\
		& = &
		\ind(J^{-1})+\ind (S^-(\lambda)\oplus S^+(\lambda))+\ind J_{\lambda}\\
		& \stackrel{(IV.2)}{=} &
		\ind (S^-(\lambda)\oplus S^+(\lambda))-d\\
		& = &
		\ind S^-(\lambda)+\ind S^+(\lambda))-d
		\stackrel{(IV.3)}{=}
		n+r-d=0.
	\end{eqnarray*}

	(V) Keeping $\tau>0$ fixed consider the family of operators
	\begin{align*}
		S(\lambda)^\ast:D(S(\lambda)^\ast)&\to L^{\infty}[-\tau,\tau],&
		[S(\lambda)^\ast y](t)&:=\dot{y}(t)+A(t,\lambda)^Ty(t)
	\end{align*}
	for a.a.\ $t\in[-\tau,\tau]$ on the domain
	\begin{equation*}
		D(S(\lambda)^\ast)
		:=
		\set{u\in W^{1,\infty}[-\tau,\tau]\mid u(-\tau)\in N(P_\lambda^-(-\tau))^{\perp},u(\tau)\in R(P_\lambda^+(\tau))^{\perp}}.
	\end{equation*}
	Then the diagram
	\begin{align*}
		\xymatrix{
		W^{1,\infty}(\R)\ar[r]^-{T(\lambda)^\ast}& L^{\infty}(\R)\ar[d]^{p}\\
		D(S(\lambda)^{\ast})\ar[r]^-{S(\lambda)^\ast}\ar[u]_{i_{\lambda}^{\ast}}& L^{\infty}[-\tau,\tau]
		}
	\end{align*}
	commutes, where $i_{\lambda}^{\ast}$ is defined according to \eqref{diagram-i}. As above, $S(\lambda)^{\ast}$ is Fredholm of index $0$ and $i_{\lambda}(N(S(\lambda)^{\ast}))=N(T(\lambda)^{\ast})$ with $\dim N(S(\lambda)^{\ast})=\dim N(T(\lambda)^{\ast})$.

	(VI) \underline{Claim}: \emph{$N(S(\lambda)^{\ast})\cap R(S(\lambda))=\{0\}$}.\\
	If $g\in R(S(\lambda))$ with preimage $\phi\in D(S(\lambda)^\ast)$ and $u\in N(S(\lambda)^{\ast})$, then
\begin{align*}
&\int_{-\tau}^{\tau}\sprod{u(s),g(s)}\d s=\int_{-\tau}^{\tau}\sprod{u(s),\dot{\phi}(s)-A(t,\lambda)\phi(s)}\d s\\
&=\int_{-\tau}^{\tau}\sprod{u(s),\dot{\phi}(s)}-\sprod{u(s),A(t,\lambda)\phi(s)}\d s\\
&=\int_{-\tau}^{\tau}\sprod{u(s),\dot{\phi}(s)}+\sprod{-u(s)A(t,\lambda)^T,\phi(s)}\d s\\
&=\int_{-\tau}^{\tau}\sprod{u(s),\dot{\phi}(s)}+\sprod{\dot{u}(s),\phi(s)}\d s=\int_{-\tau}^{\tau}\frac{\d\sprod{u(s),\phi(s)}}{\d s}\d s\\
&=\sprod{u(\tau),\phi(\tau)}-\sprod{u(-\tau),\phi(-\tau)}=0,
\end{align*}
	where the last equality follows from $u(\tau)\perp \phi(\tau)$ and $u(-\tau)\perp \phi(-\tau)$, which shows that $N(S(\lambda)^{\ast})\cap R(S(\lambda))=\{0\}$. Thanks to this result and $\ind S(\lambda)=0$ together with $\dim N(S(\lambda))=\dim N(S(\lambda)^{\ast})$, we conclude $R(S(\lambda))\oplus N(S(\lambda)^{\ast})=L^{\infty}[-\tau,\tau]$.

	(VII) Repeating the arguments from Claim~III, one has $R(S(\lambda))+ W= L^{\infty}[-\tau,\tau]$, where $W:=N(S(\lambda_1)^\ast)+\ldots+N(S(\lambda_n)^\ast)$ and the parameters $\lambda_1,\ldots,\lambda_n\in [a,b]$ are the same as in \eqref{transversal-V}.
	Due to $p(N(T(\lambda)^\ast))=N(S(\lambda)^\ast)$ we conclude $p(V)=W$ with a subspace $V$ as in \eqref{transversal-V}. Thus the vector bundles
	\begin{align*}
		E(T,V)&:=\set{(\lambda,x)\in [a,b]\tm W^{1,\infty}(\R)\mid T(\lambda)x\in V},\\
		E(S,W)&:=\set{(\lambda,x)\in [a,b]\tm D(S(\lambda))\mid S(\lambda)x\in W}
	\end{align*}
	are well-defined and the following diagram commutes:
\begin{align}\label{commdiag1}
\begin{split}
\xymatrix{
E(T,V)_{\lambda}\ar[r]^-{T(\lambda)}& V\ar[d]^p\\
E(S,W)_{\lambda}\ar[u]_{(i_E)_{\lambda}}\ar[r]^-{S(\lambda)}& W,
}
\end{split}
\end{align}
where $(i_E)_{\lambda}(x)=i_{\lambda}(x)$ and $i_{\lambda}$ is as in \eqref{diagram-i}.

	(VIII) \underline{Claim}: $\sigma(T,[a,b])=\sigma(S,[a,b])$.\\
	We observe that $\dim W\leq\dim V$ and $i_E:E(S,W)\to E(T,V)$ is an injective bundle morphism. Then $E^0:=i_E(E(S,W))$ is a subbundle of $E(T,V)$, and since $\Lambda=[a,b]$ is compact, there is a complementary bundle $E^1$ to $E^0$, i.e., $E(T,V)=E^0\oplus E^1$. We decompose $V=W_0\oplus W_1$, where $W_0:=\{\chi_{[-\tau,\tau]}u\mid\, u\in V\,\}$ and $W_1:=\{\chi_{\R\setminus[-\tau,\tau]}u\mid u\in V\,\}$. It follows from the construction of $V$ and $W$ that $p|_{W_0}:W_0\to W$ is an isomorphism, and hence $\dim W_0 =\dim W$. Now, given $\lambda\in[a,b]$ taking account the above splittings, the diagram \eqref{commdiag1} has the following form:
\begin{align*}
\begin{split}
\xymatrix{
E^0_{\lambda}\oplus E^1_{\lambda}\ar[r]^-{T(\lambda)}& W_0\oplus W_1\ar[d]^p\\
E(S,W)_{\lambda}\ar[u]_{(i_E)_{\lambda}}\ar[r]^-{S(\lambda)}& W
}
\end{split}
\end{align*}
and $T(\lambda): E^0_{\lambda}\oplus E^1_{\lambda}\to W_0\oplus W_1$ can be written as operator matrix
	\begin{equation*}
		T(\lambda)
		=
		\begin{pmatrix}
			T(\lambda)_{11}&T(\lambda)_{12}\\
			T(\lambda)_{21}&T(\lambda)_{22}
		\end{pmatrix}.
	\end{equation*}
	We prove that both mappings $T(\lambda)_{12}:E^1_{\lambda}\to W_0$ and $T(\lambda)_{21}:E^0_{\lambda}\to W_1$ are trivial. Indeed, take $u\in E^0_{\lambda}$. Since there exists a $v\in E(S,W)_{\lambda}$ with $(i_E)_{\lambda}v=u$, it follows that $T(\lambda)u$ admits the property $(T(\lambda)u)(t)=0$ for all $t\in\R\setminus[-\tau,\tau]$, which yields $T(\lambda)u\in W_0$. But $T(\lambda)u=(T(\lambda)_{11}u,T(\lambda)_{21}u)$, and therefore $T(\lambda)_{21}$ is trivial.

As for $T(\lambda)_{12}$, take $w\in R(T(\lambda)_{12})\subseteq W_0$. Then there is $u\in E^1_\lambda\subset W^{1,\infty}(\R)$ such that $T(\lambda)_{12} u=w$. Since $T(\lambda)_{12}u\in W_0$, it follows that $\dot{u}(t)-A(t,\lambda)u(t)=0$ for all $t\in(-\infty,-\tau)$ and for $t\in(\tau,\infty)$, which particularly shows that
$u(-\tau)\in N(P_\lambda^-(-\tau))$ and $u(\tau)\in R(P_\lambda^+(\tau))$. Hence, there exists a $v\in E(S,W)_\lambda$ such that $(i_E)_\lambda(v)=u$, which implies that $u\in E^0_{\lambda}$. Thus, $u\in E^0_{\lambda}\cap E^1_{\lambda}=\{0\}$, and hence $w=T(\lambda)_{12}0=0$, establishing $T(\lambda)_{12}\equiv 0$.

Thus we have proved that $T(\lambda): E(T,V)_{\lambda}\to V$ allows the decomposition:
\begin{equation}\label{decomp-S}
T(\lambda)=T(\lambda)_{11}\oplus T(\lambda)_{22}: E^0_{\lambda}\oplus E^1_{\lambda}\to W_0\oplus W_1=V.
\end{equation}
	Now, we have
	\begin{itemize}
		\item $N(T(\lambda)_{22})=\{0\}$ since $N(T(\lambda))=(i_E)_{\lambda}\bigl(N(S(\lambda))\bigr)\subseteq E^0_{\lambda}$,
		\item $\dim W=\dim E(S,W)_{\lambda}$ because $\ind S(\lambda)=0$ and $E(S,W)_{\lambda}=S(\lambda)^{-1}(W)$,
		\item $\dim W_0\oplus W_1=\dim E^0_{\lambda}\oplus E^1_{\lambda}$ since $\ind T(\lambda)=0$,
		\item $\dim E^0_{\lambda}=\dim (i_E)_{\lambda}(E(S,V))_{\lambda}=\dim E(S,V)_{\lambda}$ because $(i_E)_{\lambda}$ is injective.
	\end{itemize}
	Thus, $\dim E^0_{\lambda}=\dim W_0<\infty$, $\dim E^1_{\lambda}=\dim W_1<\infty$, and $T(\lambda)_{11}: E^0_{\lambda}\to W_0$ and $T(\lambda)_{22}:E^1_{\lambda}\to W_1$ are Fredholm of index $0$. Moreover, from $N(T(\lambda)_{22})=\{0\}$ results that $T(\lambda)_{22}$ is an isomorphism and by means of \eqref{decomp-S} and \lref{lemparity} and \ref{lemparity2}, we obtain
\begin{align}\label{calculations-I}
\begin{split}
\sigma(T,[a,b])&=\sigma(T\circ\hat T,[a,b])
=
\sigma((T_{11}\oplus T_{22})\circ (\hat T_1\oplus \hat T_2),[a,b])
\\&=\sigma(T_{11}\circ\hat T_1,[a,b])\cdot\sigma(T_{22}\circ\hat T_2,[a,b])\\
&=\sigma(T_{11}\circ\hat T_1,[a,b])\cdot 1
=
\sigma(T_{11}\circ\hat T_1,[a,b]),
\end{split}
\end{align}
where $\hat T: [a,b]\tm V\to E(T,V)$, $\hat T_1: [a,b]\tm W_0\to E(T_{11},W_0)$ and $\hat T_2: [a,b]\tm W_1\to E(T_{22},W_1)$ are arbitrary bundle trivializations. If we consider the commutative diagram
	\begin{equation*}
		\xymatrix{
		E^0_{\lambda}\ar[r]^-{T(\lambda)_{11}}& W_0\ar[d]^{p|_{W_0}}_{\cong}\\
		E(\tilde{T},W)_{\lambda}\ar[u]^{(i_E)_{\lambda}}_{\cong}\ar[r]^-{S(\lambda)}& W,
		}
	\end{equation*}
	then again \lref{lemparity} and \ref{lemparity2} imply that
\begin{align}\label{calculations-II}
\sigma(T_{11}\circ\hat T_1,[a,b])=\sigma(S\circ\hat T_0,[a,b])=\sigma(S,[a,b]),
\end{align}
where also $\hat T_0: [a,b]\tm W\to E(S,W)$ is an arbitrary bundle trivialization. Finally, taking into account \eqref{calculations-I} and \eqref{calculations-II}, we derive the claimed equality.

	(IX) \underline{Claim}: \emph{$\sigma(S,[a,b])=\sigma(Q,[a,b])$ for the operator
	\begin{equation*}
		Q(\lambda):D(Q(\lambda))\subset W^{1,\infty}[-\tau,\tau]\to L^{\infty}[-\tau,\tau],\quad [Q(\lambda)u](t)=\dot{u}(t)
	\end{equation*}
	defined on the domain
	\begin{align}\label{Qdomains}
		D(Q(\lambda))=\set{u\in W^{1,\infty}[-\tau,\tau]\mid u(-\tau)\in N(P_\lambda^-(0)),\, u(\tau)\in R(P_\lambda^+(0))}.
	\end{align}}
	We abbreviate $\Phi_\lambda(t):=\Phi_\lambda(t,0)$, $S_\lambda:=S(\lambda)$ and similarly for further paths. With
	\begin{align*}
		M(\lambda)&\in GL(W^{1,\infty}[-\tau,\tau]),&
		[M(\lambda) u](t)=\Phi_\lambda(t)u(t)\fall t\in [-\tau,\tau],
	\end{align*}
	we observe that $M_\lambda^{-1}S_\lambda M_\lambda$ are Fredholm of index $0$ (cf.~(IV.5)) and defined on
	\begin{align*}
&D(M_\lambda^{-1} S_\lambda M_\lambda)=\{M_\lambda^{-1}u\in W^{1,\infty}[-\tau,\tau]: u\in D(S_\lambda)\,\}\\
&=\{M_\lambda^{-1}u\in W^{1,\infty}[-\tau,\tau] : u(-\tau)\in N(P_\lambda^-(-\tau)),\, u(\tau)\in R(P_\lambda^+(\tau))\,\}\\
&=\{v\in W^{1,\infty}[-\tau,\tau] : (M_\lambda v)(-\tau)\in N(P_\lambda^-(-\tau)),\, (M_\lambda v)(\tau)\in R(P_\lambda^+(\tau))\,\}\\
&=\set{v\in W^{1,\infty}[-\tau,\tau]:\, v(-\tau)\in N(P_\lambda^-(0)), v(\tau)\in R(P_\lambda^+(0))},
	\end{align*}
	where we used $\Phi_\lambda(t)P_\lambda^\pm(0)\phi_\lambda(t)^{-1}=P_{\lambda}^{\pm}(t)$ for $t\in\R_\pm$ in the last equality (see \eqref{ed0}). Moreover, for $u\in D(M_\lambda^{-1} S_\lambda M_\lambda)$ one has the identity
	\begin{align*}
		[M_\lambda^{-1} S_\lambda M_\lambda u](t)
		&\equiv
		\Phi_\lambda(t)^{-1}(\dot{\Phi}_\lambda(t)u(t)+\Phi_\lambda(t)\dot{u}(t)-A(t,\lambda)\Phi_\lambda(t)u(t))\\
		&\equiv
		\dot{u}(t)+\Phi_\lambda(t)^{-1}(\dot{\Phi}_\lambda(t)-A(t,\lambda)\Phi_\lambda(t))u(t)
		\equiv
		\dot{u}(t)
	\end{align*}
	a.e.\ on $\R$, i.e.\ $M_\lambda^{-1} S_\lambda M_\lambda=Q(\lambda)$. Hence, we obtain from \lref{lemparity}(c) that
	\begin{align*}
		\sigma(S,[a,b])
		&=
		\sigma(M^{-1},[a,b])\cdot\sigma(S,[a,b])\cdot\sigma(M,[a,b])\\
		&=\sigma(M^{-1}S M,[a,b])=\sigma(Q,[a,b]).
	\end{align*}

	(X) It is not hard to see that $L^{\infty}[-\tau,\tau]=Y_0\oplus Y_1$, where $Y_0$ is the $d$-dimensional space of constant $\R^d$-valued functions and
	\begin{align*}
		Y_1=\left\{u\in L^{\infty}[-\tau,\tau] : \int^{\tau}_{-\tau}u(s)\d s=0\,\right\}.
	\end{align*}
What is more, let us observe that $Y_0$ is transversal to the image of $Q$, i.e.,
	\begin{equation}\label{transversal-II}
		R(Q(\lambda))+Y_0=L^{\infty}[-\tau,\tau].
	\end{equation}
	Indeed, let $u\in L^{\infty}[-\tau,\tau]$. If we define $c,v:[-\tau,\tau]\to\R^d$ as
	\begin{align*}
		c(t)&:\equiv\frac{1}{2\tau}\int_{-\tau}^{\tau}u(s)\d s,&
		v(t)&:=\int^{t}_{-\tau}\left(u(s)-c(s)\right)\d s,
	\end{align*}
	then $v$ belongs to $D(Q(\lambda))$, defined in \eqref{Qdomains} for all $\lambda\in [a,b]$, and
	\begin{equation*}
		[Q(\lambda)v](t)+c(t)=u(t)-c(t)+c(t)=u(t)\fall t\in[-\tau,\tau]
	\end{equation*}
	proves \eqref{transversal-II}. Thus $E(Q,Y_0)$ is well-defined with the fibers
	\begin{align*}
		&
		E(Q,Y_0)_\lambda
		=
		Q(\lambda)^{-1}Y_0=\{u\in D(Q(\lambda)): \dot{u}(t)\equiv\text{constant}\}\\
		&=
		\set{u_\xi^\eta:\,u_\xi^\eta(t)=\tfrac{1}{2}\left(1+\tfrac{t}{\tau}\right)\eta+\tfrac{1}{2}\left(1-\tfrac{t}{\tau}\right)\xi\text{ for }\xi\in N(P_\lambda^-(0)),\,\eta\in R(P_\lambda^+(0))}.
	\end{align*}
Moreover, $Q(\lambda)$ acts on the fibers $E(Q,Y_0)_\lambda$ into $Y_0$ by $Q(\lambda)u_\xi^\eta=\frac{1}{2\tau}(\eta-\xi)$. Now we are in a position to introduce the following commutative diagram:
\begin{align*}
\xymatrix{
E(Q,Y_0)_{\lambda}\ar[r]^{Q(\lambda)}\ar[d]_{e_{\lambda}}^{\cong}& Y_0\ar[d]^m_{\cong}\\
N(P_\lambda^-(0))\oplus R(P_\lambda^+(0)\ar[r]^-{\hat L_{\lambda}}&\R^d,
}
\end{align*}
	where the mappings $e_{\lambda}$, $\hat L_{\lambda}$ and $m$ are defined as follows
	\begin{align*}
		e_\lambda:E(Q,Y_0)_{\lambda}&\to N(P_\lambda^-(0))\oplus R(P_\lambda^+(0),&
		e_{\lambda}(u)&:=(u(-\tau),u(\tau))\\
		m:Y_0&\to \R^d,&
		m(u)&:=2\tau u\\
		\hat L_{\lambda}:N(P_\lambda^-(0))\oplus R(P_\lambda^+(0)&\to \R^d,&
		\hat L_{\lambda}(u,v)&=v-u.
	\end{align*}
	Hence, in view of \lref{lemparity} and \ref{lemparity2}, we deduce the desired conclusion
	\begin{equation}
		\sigma(Q,[a,b])=\sigma(Q\circ\hat T,[a,b])=\sigma(\hat L\circ\hat T^{\hat L},[a,b]),
		\label{star2}
	\end{equation}
	where $\hat T: [a,b]\tm Y_0\to E(Q,Y_0)$ is an arbitrary bundle trivialization with second bundle trivialization $\hat T^{\hat L}:[a,b]\tm (\R^{d-m^+}\tm\set{0}\oplus\set{0}\tm\R^{m^-})\to \mathrm{N}(P^-(0))\oplus\mathrm{R}(P^+(0))$,
	\begin{equation*}
		\hat T^{\hat L}(\lambda,x,y)
		:=
		\intoo{\sum_{i=1}^{d-m^+} x_i \xi^+_i(\lambda), \sum_{j=1}^{m^-} y_j\xi^-_j(\lambda)}
	\end{equation*}
	and the functions $\xi^+_i$, $\xi^-_j$ from \pref{propetagamma} with $m^+=m^-$.

	(XI) \underline{Claim}: $\sigma(\hat L\circ\hat T^{\hat L},[a,b])=\sgn E(a)\cdot \sgn E(b)$.\\
	From \lref{lemparity} it follows that
	$
		\sigma(\hat L\circ\hat T^{\hat L},[a,b])
		=
		\sgn\det(\hat L_a\circ\hat T^{\hat L}_a)\cdot \sgn \det(\hat L_b\circ\hat T^{\hat L}_b),
	$
	where using \pref{propetagamma} one has for $\lambda\in\set{a,b}$ that
\begin{align*}
	&
	\det(\hat L_\lambda\circ\hat T^{\hat L}_\lambda)
	=
	\det(-\xi^+_1(\lambda),\ldots,-\xi_{d-m^+}^+(\lambda),\xi^-_1(\lambda),\ldots,\xi_{m^-}^-(\lambda))\\
	&=
	(-1)^{d-m^+}\det (\xi^+_1(\lambda),\ldots,\xi_{d-m^+}^+(\lambda),\xi^-_1(\lambda),\ldots,\xi_{m^-}^-(\lambda))
	=
	(-1)^{d-m^+} E(\lambda)
\end{align*}
and thus
	$
		\sigma(\hat L\circ\hat T^{\hat L},[a,b])
		=
		(\sgn (-1)^{d-m^+})^2\sgn E(a)\sgn E(b)
		=
		\sgn E(a)\sgn E(b).
	$

	(XII) Finally, taking into account the previous steps, one obtains
	\begin{eqnarray*}
		\sigma(T,[a,b])
		& \stackrel{\text{(VIII)}}{=} &
		\sigma(S,[a,b])
		\stackrel{\text{(IX)}}{=}
		\sigma(Q,[a,b])
		\stackrel{\eqref{star2}}{=}
		\sigma(\hat L\circ\hat T^{\hat L},[a,b])\\
		& \stackrel{\text{(XI)}}{=} &
		\sgn E(a)\cdot \sgn E(b),
	\end{eqnarray*}
	which completes the proof of (a).

	(b) The arguments from the above proof of part (a) carry over to the present situation with the spaces $W^{1,\infty}(\R)$ and $L^\infty(\R)$ replaced by $W_0^{1,\infty}(\R)$ resp.\ $L_0^\infty(\R)$, provided the respective statements (b) of \tref{thmgprop}, \ref{thmadmin} and \ref{thmfred} are employed.
\end{proof}

We conclude this section with a local version of \tref{thmmain} involving the parity index $\sigma(T,\lambda^\ast)$ (see App.~\ref{appA}). Thereto, we say that an Evans function $E$ for \eqref{var} \emph{changes sign} at a parameter value $\lambda^\ast\in\Lambda^\circ$, if there exists a neighborhood $\Lambda_0\subseteq\tilde\Lambda$ of $\lambda^\ast$ so that $E(\lambda)\neq 0$ for all $\lambda\in\Lambda_0\setminus\set{\lambda^\ast}$ and $\lim_{\eps\searrow 0}\sgn E(\lambda^\ast-\eps)\cdot \sgn E(\lambda^\ast+\eps)=-1$ hold. Then the Intermediate Value Theorem yields $E(\lambda^\ast)=0$. Moreover, for smooth Evans functions a sign change occurs, if $\lambda^\ast$ is a zero of odd order.

\begin{corollary}[Evans function and parity index]\label{corEv}
	If an Evans function $E$ of \eqref{var} changes sign at $\lambda^\ast$, then $\sigma(T,\lambda^\ast)=-1$.
\end{corollary}
\begin{proof}
	By assumption there is a neighborhood $\Lambda_0$ of $\lambda^\ast$ such that $E(\lambda)\neq 0$ on $\Lambda_0\setminus\set{\lambda^\ast}$ and hence \tref{thmadmin} combined with \pref{propEvans} yield that $T(\lambda)$ are nonsingular for $\lambda\neq\lambda^\ast$. Therefore, \tref{thmmain} implies
	\begin{equation*}
		\sigma(T,[\lambda^\ast-\eps,\lambda^\ast+\eps])
		=
		\sgn\sigma(T,[\lambda^\ast-\eps,\lambda^\ast+\eps])
		=
		\sgn E(\lambda^\ast-\eps)\cdot\sgn E(\lambda^\ast+\eps)
	\end{equation*}
	and passing to the limit $\eps\searrow 0$ yields the claim.
\end{proof}

\begin{remark}[multiplicities]\label{remmult}
	With the closed operators
	\begin{equation*}
		T(\lambda):D(T(\lambda))\subseteq L^\infty(\R)\to L^\infty(\R)\fall\lambda\in\Lambda
	\end{equation*}
	on the domains $D(T(\lambda)):=W^{1,\infty}(\R)$ (or with the spaces $L_0^\infty(\R)$ and $W_0^{1,\infty}(\R)$, resp.) it is due to \cite{minh:94} that the dichotomy spectrum $\Sigma(\lambda)$ of \eqref{var} is related to the spectrum $\sigma(T(\lambda))\subseteq\C$ of the operator $T(\lambda)$ via $\Sigma(\lambda)=\sigma(T(\lambda))\cap\R$.

	Hypothesis $(H_2)$ and $E(\lambda^\ast)=0$ yield the inclusion $0\in\Sigma(\lambda^\ast)$ and we denote the spectral interval containing $0$ as \emph{critical}. More precisely, $0$ is an eigenvalue of $T(\lambda^\ast)$ with \emph{geometric multiplicity} $\mu_g:=\dim N(T(\lambda^\ast))=\dim\bigl(R(P_{\lambda^\ast}^+(0))\cap N(P_{\lambda^\ast}^-(0))\bigr)$. In relation to the algebraic multiplicity $\mu$ of the critical spectral interval it is $\mu_g\leq\mu\leq d$.
\end{remark}
\section{Bifurcation in Carath{\'e}odory equations}
\label{sec4}
We now establish Evans functions as a tool to detect bifurcations of bounded entire solutions to Carath{\'e}odory equations \eqref{cde}. An entire solution $\phi^\ast=\phi_{\lambda^\ast}$ of $(C_{\lambda^\ast})$ is said to \emph{bifurcate} at the parameter $\lambda^\ast\in\tilde\Lambda$, if there exists a sequence $(\lambda_n)_{n\in\N}$ in $\tilde\Lambda$ converging to $\lambda^\ast$ such that each $(C_{\lambda_n})$ has a bounded entire solution $\psi_n\neq\phi_{\lambda_n}$ with
\begin{equation*}
	\lim_{n\to\infty}\sup_{t\in\R}\abs{\psi_n(t)-\phi^\ast(t)}=0.
\end{equation*}
In other words, $\phi^\ast$ is an accumulation point of bounded entire solutions not contained in the family $(\phi_\lambda)_{\lambda\in\tilde\Lambda}$. Given $\Lambda\subseteq\tilde\Lambda$, the subset $\fB_\Lambda$ of all parameters $\lambda\in\Lambda$ such that there occurs a bifurcation at $(\phi_\lambda,\lambda)$ is denoted as \emph{set of bifurcation values} for \eqref{cde}.

\begin{theorem}[necessary bifurcation condition]
	\label{thmnobif}
	Let $\lambda^\ast\in\tilde\Lambda$ and suppose that $(H_0$--$H_1)$ hold. If $\lambda^\ast\in \fB_{\tilde\Lambda}$, i.e.\ the bounded, permanent, entire solution $\phi^\ast$ of $(C_{\lambda^\ast})$ bifurcates at $\lambda^\ast$, then $\phi^\ast$ is not hyperbolic on $\R$, i.e.\ $0\in\Sigma(\lambda^\ast)$.
\end{theorem}
\begin{proof}
	This consequence of the Implicit Function Theorem \cite[pp.~7--8, Thm.~I.1.1]{kielhoefer:12} is established akin to \cite[Thm.~3.8]{poetzsche:11} in the context of ordinary differential equations.
\end{proof}

Our appproach requires Hypotheses $(H_0$--$H_2)$ to hold with a parameter space $\tilde\Lambda\subseteq\R$ containing a neighborhood of $\lambda^\ast$. Then \pref{propetagamma} guarantees the existence of an Evans function $E:[\lambda^{\ast}-\bar\eps,\lambda^{\ast}+\bar\eps]\to\R$ for the variation equation \eqref{var}.

A combination of \tref{thmgprop} with \pref{propEvans} yields the implications
\begin{equation*}
	\lambda^\ast\in\fB_\Lambda
	\quad\Rightarrow\quad
	0\in\Sigma(\lambda^\ast)
	\quad\Leftrightarrow\quad
	E(\lambda^\ast)=0,
\end{equation*}
while the converse holds when $E$ has an actual sign change at $\lambda^\ast$. Note that we impose no further assumption and thus extend the sufficient bifurcation conditions from \cite{poetzsche:12}, which were limited to critical spectral intervals containing a geometrically simple eigenvalue~$0$. For the sake of a compact notation in the next result we introduce the \emph{prescribed branch}
\begin{equation*}
	\cS:=\{(\phi_{\lambda},\lambda)\in W^{1,\infty}(\mathbb{R},\Omega)\times \tilde\Lambda\mid \phi_{\lambda} \text{ is as in }(H_1)\}
\end{equation*}
of solutions to \eqref{cde} and its subset $\cS_{\mathrm{H}}:=\set{(\phi_\lambda,\lambda)\in \cS\mid\phi_\lambda\text{ is hyperbolic on }\R}$.
\begin{theorem}[bifurcation of bounded solutions homoclinic to $\cS$]
	\label{thmsingle}
	Let $(H_0$--$H_2)$ hold with Morse indices $m^+=m^-$. If an Evans function $E:[\lambda^{\ast}-\bar\eps,\lambda^{\ast}+\bar\eps]\to\R$ for \eqref{var} changes sign at $\lambda^\ast$, then the entire solution $\phi^\ast$ to $(C_{\lambda^\ast})$ bifurcates at $\lambda^\ast$ in the following sense:
	\begin{enumerate}
		\item[(a)] There exists a $\delta_0>0$ such that for each $\delta\in(0,\delta_0)$ there is a connected component
	\begin{equation*}
		\cC
		\subseteq
		\set{(\phi,\lambda)\in W^{1,\infty}(\R)\tm\Omega\mid\phi:\R\to\Omega\text{ solves }\eqref{cde}}
		\setminus \cS_{\mathrm{H}}
	\end{equation*}
	containing the pair $(\phi^\ast,\lambda^\ast)$, which joins the complement $\cS\setminus\cS_{\mathrm{H}}$ with the set $\bigl\{(x,\lambda)\in W^{1,\infty}(\R)\tm\Lambda\mid\norm{x-\phi_\lambda}_{1,\infty}=\delta\bigr\}$ (see Fig.~\ref{figsingle}).

		\item[(b)] For each $(\phi,\lambda)\in\cC$ the bounded entire solution $\phi:\R\to\Omega$ is homoclinic to $\phi_\lambda$.
	\end{enumerate}
\end{theorem}
\begin{SCfigure}[2]
	\includegraphics[scale=0.5]{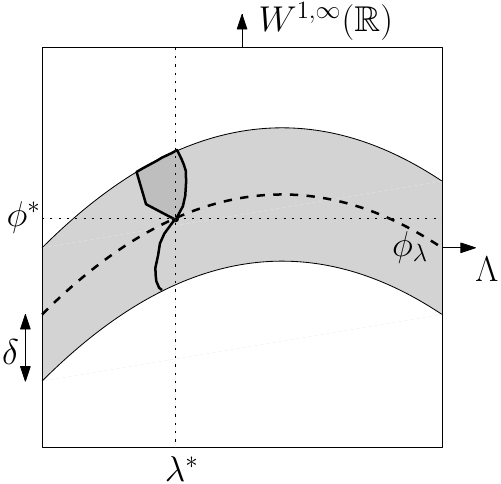}
	\caption{To \tref{thmsingle}: At some pair $(\phi^\ast,\lambda^\ast)\in W^{1,\infty}(\R)\tm\Lambda$ a continuum of solutions homoclinic to $\phi_\lambda$ (dark grey shaded) bifurcates from the prescribed branch $\cS$ (dashed line) connecting it to the tube $\bigl\{(x,\lambda)\in W^{1,\infty}(\R)\tm\Lambda\mid\norm{x-\phi_\lambda}_{1,\infty}=\delta\bigr\}$ (light grey shaded) for sufficiently small $\delta>0$}
	\label{figsingle}
\end{SCfigure}

Using the examples below, it is not hard to see that a sign change of an Evans function is a sufficient, but not a necessary condition for bifurcation in the sense of \tref{thmsingle}. Furthermore, the fact that $E(\lambda^\ast)=0$ and \cref{corEvans}(c) imply $d>1$, i.e.\ \tref{thmsingle} does not apply to scalar Carath{\'e}odory equations~\eqref{cde} (where $d=1$).
\begin{proof}
	Above all, $\phi_\lambda\in W^{1,\infty}(\R,\Omega)$ for each $\lambda\in\tilde\Lambda$ holds due to \tref{thmchar}.
	
	We apply the abstract bifurcation criterion in \tref{thmbifurcation} to \eqref{abs} with the parametri\-zed operator $G$ from \tref{thmgprop} and the Banach spaces $X=W_0^{1,\infty}(\R)$, $Y=L_0^\infty(\R)$. Indeed, because of \tref{thmgprop}(b) the mapping $G:U\to L_0^\infty(\R)$ is well-defined and continuous on a product $U:=\bigl\{x\in W_0^{1,\infty}(\R):\,\norm{x}_\infty<\rho\bigr\}^\circ\tm\Lambda$. Furthermore, the partial derivative $D_1G:U\to L(W_0^{1,\infty}(\R),L_0^\infty(\R))$ exists as continuous function, $G(0,\lambda)\equiv 0$ holds on $\Lambda$, while \tref{thmfred}(b) shows that $\lambda\mapsto D_1G(0,\lambda)$ defines a path of Fredholm operators with index $0$. Since an Evans function $E$ is assumed to change sign at $\lambda^\ast$, we readily derive from \cref{corEv} that $\sigma(D_1G(0,\cdot),\lambda^\ast)=-1$ holds. Consequently, \tref{thmbifurcation} (with $\lambda_0=\lambda^\ast$) shows that $(0,\lambda^\ast)$ is a bifurcation point of \eqref{abs} and that there exists a $\delta_0>0$ and a connected set of nonzero solutions to \eqref{abs} in $W_0^{1,\infty}(\R)$ emanating from $(0,\lambda^\ast)$ to the surface $\norm{x}_{1,\infty}=\delta$ for $\delta\in(0,\delta_0)$. Note in this context that \tref{thmnobif} guarantees the equivalence $D_1G(0,\lambda)\in GL(W_0^{1,\infty}(\R),L_0^\infty(\R))\Leftrightarrow(\phi_\lambda,\lambda)\in\cS_H$.

	Given this, \tref{thmchar}(b) implies that the bounded entire solution $\phi^\ast$ to $(C_{\lambda^\ast})$ bifurcates at $\lambda=\lambda^\ast$ into a set of solutions to \eqref{cde} being homoclinic to $\phi_\lambda$. In particular, the statements on the continuum of bifurcating bounded entire solutions to \eqref{cde} holds.
\end{proof}

The following example illustrates the generality and applicability of \tref{thmsingle}.
\begin{example}\label{example-9}
	Let $n\in\N$, $\alpha>0$ and $\tilde\Lambda=\R$. Consider a Carath{\'e}odory equation \eqref{cde} in $\Omega=\R^{2n}$ with right-hand side $f:\R\tm\R^{2n}\tm\R\to\R^{2n}$,
	\begin{align*}
		f(t,x,\lambda)
		&:=
		\begin{pmatrix}
			a(t)I_n & 0\\
			C(\lambda) & -a(t)I_n
		\end{pmatrix}x
		+F(t,x,\lambda),&
		a(t)&:=
		\begin{cases}
			-\alpha,&t\geq 0,\\
			\alpha,&t<0
		\end{cases}
	\end{align*}
	with a continuous function $C:\R\to\R^{n\tm n}$ and a nonlinearity $F:\R\tm\R^d\tm\R\to\R^d$ such that the resulting right-hand side $f$ might fulfill both Hypothesis $(H_0)$ and
	\begin{align}
		F(t,0,\lambda)&\equiv 0,&
		D_2F(t,0,\lambda)&\equiv 0\on\R\tm\R.
		\label{cal01}
	\end{align}
	Consequently, \eqref{cde} has the trivial solution for all parameters $\lambda\in \R$, i.e., we can choose the continuous branch $\phi_\lambda(t):\equiv 0$ on $\R$ and assumption $(H_1)$ holds. For each $\gamma\in\R$ the shifted variation equation \eqref{var} along the trivial solution becomes
	\begin{equation}
		\dot x
		=
		\begin{pmatrix}
			(a(t)-\gamma)I_n & 0_n\\
			C(\lambda) & (-a(t)-\gamma)I_n
		\end{pmatrix}x.
		\label{varshift}
	\end{equation}
	We first determine the dichotomy spectrum $\Sigma(\lambda)$ of \eqref{var}. Thereto, note that \eqref{varshift} is piecewise autonomous which on the respective semiaxes $\R_+$ and $\R_-$ becomes
	\begin{align*}
		\dot x&=
		\begin{pmatrix}
			(-\alpha-\gamma)I_n & 0_n\\
			C(\lambda) & (\alpha-\gamma)I_n
		\end{pmatrix}x,&
		\dot x&=
		\begin{pmatrix}
			(\alpha-\gamma)I_n & 0_n\\
			C(\lambda) & (-\alpha-\gamma)I_n
		\end{pmatrix}x.
	\end{align*}
	In case $\gamma<-\alpha$ it is $-\alpha-\gamma>0$, $\alpha-\gamma>0$ and thus \eqref{varshift} has an exponential dichotomy with projector $P(t)\equiv I_{2n}$ on $\R$. In case $\gamma>\alpha$ it holds $-\alpha-\gamma<0$, $\alpha-\gamma<0$ and \eqref{varshift} is exponentially dichotomic with projector $P(t)\equiv 0_{2n}$ on $\R$. In conclusion, this implies that $\Sigma(\lambda)\subseteq[-\alpha,\alpha]$. For $\gamma\in\set{-\alpha,\alpha}$ one sees that \eqref{varshift} has nontrivial bounded entire solutions, which yields $\set{-\alpha,\alpha}\subseteq\Sigma(\lambda)$. In the remaining situation $\gamma\in(-\alpha,\alpha)$ the equation \eqref{varshift} has an exponential dichotomy on the semiaxis $\R_+$ with projector
	\begin{align*}
		P_\lambda^+(t)
		&\equiv\begin{pmatrix}
			I_n & 0_n\\
			-\tfrac{1}{2\alpha}C(\lambda) & 0_n
		\end{pmatrix},&
		R(P_\lambda^+(0))
		&=
		\set{\binom{\xi}{\eta}\in\R^{2n}:\,\eta=-\tfrac{1}{2\alpha}C(\lambda)\xi}
	\end{align*}
	(and Morse index $m^+=n$), as well as on the the semiaxis $\R_-$ with projector
	\begin{align*}
		P_\lambda^-(t)
		&\equiv\begin{pmatrix}
			0_n & 0_n\\
			-\tfrac{1}{2\alpha}C(\lambda) & I_n
		\end{pmatrix},&
		N(P_\lambda^-(0))
		&=
		\set{\binom{\xi}{\eta}\in\R^{2n}:\,\eta=\tfrac{1}{2\alpha}C(\lambda)\xi}
	\end{align*}
	(and Morse index $m^-=n$). Therefore, $m^-=m^+$ and according to \pref{propetagamma} a globally defined Evans function for the variation equation \eqref{var} can be constructed as
	\begin{align*}
		E:\R&\to\R,&
		E(\lambda)
		&=
		\det\begin{pmatrix}
			I_n & I_n\\
			-\tfrac{1}{2\alpha}C(\lambda) & \tfrac{1}{2\alpha}C(\lambda)
		\end{pmatrix}
		=
		\frac{\det C(\lambda)}{\alpha^n},
	\end{align*}
	which results in the following two observations:

	(1) The equation \eqref{varshift} having a nontrivial bounded entire solution is equivalent to
	\begin{equation*}
		R(P_\lambda^+(0))\cap N(P_\lambda^-(0))\neq\set{0}
		\quad\Leftrightarrow\quad
		N(C(\lambda))\neq\set{0}
		\quad\Leftrightarrow\quad
		E(\lambda)\neq 0,
	\end{equation*}
	which leads to the dichotomy spectrum
	\begin{equation*}
		\Sigma(\lambda)
		=
		\begin{cases}
			[-\alpha,\alpha],&\lambda\in E^{-1}(0),\\
			\set{-\alpha}\cup\set{\alpha},&\lambda\not\in E^{-1}(0)
		\end{cases}
	\end{equation*}
	of the variation equation \eqref{var} (cf.~\cref{corEvans}). In detail, if an Evans function $E$ has an isolated zero $\lambda^\ast\in\R$, then the critical spectral interval $\Sigma(\lambda^\ast)=[-\alpha,\alpha]$ of algebraic multiplicity $2n$ splits into two spectral intervals $\set{-\alpha},\set{\alpha}$ (in fact singletons) of algebraic multiplicity $n$ for $\lambda\neq\lambda^\ast$. Here, the critical spectral interval $\Sigma(\lambda^\ast)$ consists of eigenvalues to the operator $T(\lambda^\ast)$ from \rref{remmult} with geometric multiplicity $\dim N(C(\lambda^\ast))$.

	(2) If $E:\R\to\R$ even changes sign at $\lambda^\ast$, then \tref{thmsingle} implies for any non\-lin\-earity $F$ satisfying \eqref{cal01} that nontrivial bounded entire solution to \eqref{cde} being homoclinic to $0$ bifurcate at $\lambda^\ast$ from the zero branch, i.e.\ $\fB_\R=\set{\lambda\in\R:\,E\text{ changes sign at }\lambda}$. Note that the bifurcation criteria from \cite{poetzsche:12} do not apply to \eqref{cde} unless $\dim N(C(0))=1$.
\end{example}

Preparing further examples we introduce a prototypical equation:
\begin{lemma}\label{lemproto}
	Let $\nu,\mu\in\R$. The general solution of the ordinary differential equation
	\begin{equation}
		\dot x
		=
		\begin{pmatrix}
			-\tanh t & 0 \\
			0 & \tanh t
		\end{pmatrix}x
		+
		\begin{pmatrix}
			0\\
			\nu x_1^2
		\end{pmatrix}
		+
		\begin{pmatrix}
			0\\
			\mu
		\end{pmatrix}
		\label{ode1}
	\end{equation}
	satisfies for all $t\in\R$ and initial values $\xi\in\R^2$ that
	\begin{equation*}
		\varphi(t;0,\xi)
		=
		\begin{pmatrix}
			\tfrac{1}{\cosh t}\xi_1\\
			\tfrac{\nu\xi_1^2}{2}\tanh t+
			\cosh t\bigl[\xi_2+(\nu\xi_1^2+2\mu)\arctan\tanh\tfrac{t}{2}\bigr]
		\end{pmatrix}.
	\end{equation*}
	Moreover, the equivalence $\varphi(\cdot;0,\xi)\in L^\infty(\R)\Leftrightarrow	\nu\xi_1^2+2\mu=0\text{ and }\xi_2=0$ holds.
\end{lemma}
\begin{proof}
	Because equation \eqref{ode1} is of lower triangular form the given expression for the general solution $\varphi$ is due to the Variation of Constants Formula \cite[Thm.~2.10]{aulbach:wanner:96}. In order to identify the bounded entire solutions of \eqref{ode1} we first note the limits
	\begin{equation}
		\lim_{t\to\pm\infty}\cosh t\intoo{\arctan\tanh\tfrac{t}{2}\mp\tfrac{\pi}{4}}=\mp\tfrac{1}{2}.
		\label{ode3}
	\end{equation}
	On the one hand, the representation
	\begin{align*}
		&
		\cosh t\bigl[\xi_2+(\nu\xi_1^2+2\mu)\arctan\tanh\tfrac{t}{2}\bigr]\\
		=&
		\cosh t\bigl[(\nu\xi_1^2+2\mu)\intoo{\arctan\tanh\tfrac{t}{2}-\tfrac{\pi}{4}}\bigr]
		+
		\cosh t\bigl[\xi_2+(\nu\xi_1^2+2\mu)\tfrac{\pi}{4}\bigr]
	\end{align*}
	and \eqref{ode3} guarantee that $\sup_{t\geq 0}\abs{\varphi(t;0,\xi)}<\infty$ is equivalent to
	\begin{equation}
		0
		=
		\xi_2+(\nu\xi_1^2+2\mu)\tfrac{\pi}{4}.
		\label{ode3a}
	\end{equation}
	On the other hand,
	\begin{align*}
		&
		\cosh t\bigl[\xi_2+(\nu\xi_1^2+2\mu)\arctan\tanh\tfrac{t}{2}\bigr]\\
		=&
		\cosh t\bigl[(\nu\xi_1^2+2\mu)\intoo{\arctan\tanh\tfrac{t}{2}+\tfrac{\pi}{4}}\bigr]
		+
		\cosh t\bigl[\xi_2-(\nu\xi_1^2+2\mu)\tfrac{\pi}{4}\bigr]
	\end{align*}
	combined with \eqref{ode3} ensure that $\sup_{t\leq 0}\abs{\varphi(t;0,\xi)}<\infty$ holds if and only if
	\begin{equation}
		0
		=
		\xi_2-(\nu\xi_1^2+2\mu)\tfrac{\pi}{4}.
		\label{ode3b}
	\end{equation}
	The two relations \eqref{ode3a} and \eqref{ode3b} in turn are equivalent to $\nu\xi_1^2+2\mu=0$ and $\xi_2=0$.
\end{proof}

We proceed to an example with a nontrivial continuous branch of nontrivial bounded solutions $\phi_\lambda$. It exhibits a transcritical bifurcation, which can also be verified in terms of the degenerate fold bifurcation from \cite[Thm.~4.2]{poetzsche:12}.
\begin{figure}[ht]
	\includegraphics[scale=0.75]{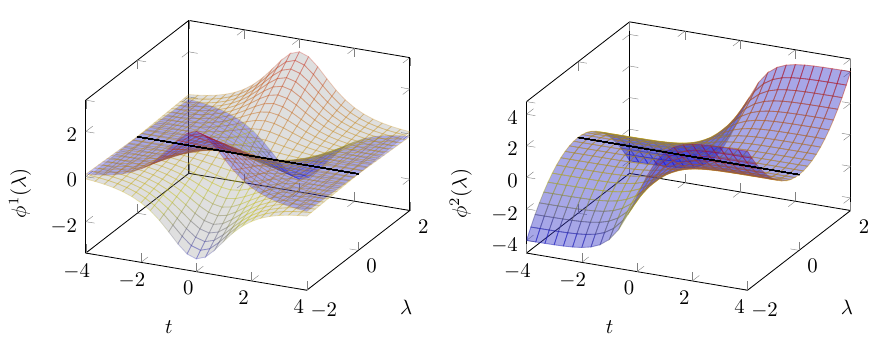}
	\caption{Bifurcation of solutions homoclinic to $\phi_\lambda$: The blue branch $\phi_\lambda^-$ bifurcates from the gray branch $\phi_\lambda=(\phi^1(\lambda),\phi^2(\lambda))$ at $\lambda^\ast=0$ and the trivial solution (black line)}
	\label{figbranches1}
\end{figure}

\begin{example}\label{ex10}
	Let $\tilde\Lambda=\R$. Consider an ordinary differential equation \eqref{cde} in $\Omega=\R^2$ with the right-hand side
	\begin{align*}
		f(t,x,\lambda)
		&:=
		\begin{pmatrix}
			-\tanh t & 0 \\
			0 & \tanh t
		\end{pmatrix}x
		+
		\begin{pmatrix}
			0\\
			x_1^2
		\end{pmatrix}
		-
		\begin{pmatrix}
			0\\
			\lambda^2
		\end{pmatrix}.
	\end{align*}
	It fits in the framework of \eqref{ode1} with $\nu=1$, $\mu=-\lambda^2$ and consequently \lref{lemproto} implies that the initial values $\xi^\pm(\lambda)=\pm\binom{\sqrt{2}\lambda}{0}$ yield two continuous branches of bounded entire solutions to \eqref{cde}. For instance, $\xi^+(\lambda)$ leads to the branch of bounded solutions
	\begin{align*}
		\phi_\lambda:\R&\to\R^2,&
		\phi_\lambda(t)
		&:=
		\lambda
		\begin{pmatrix}
			\frac{\sqrt{2}}{\cosh t}\\
			\lambda\tanh t
		\end{pmatrix}\fall\lambda\in\R.
	\end{align*}
	It intersects the branch of bounded entire solutions emanating from the initial values $\xi^-(\lambda)$ and is given by
	\begin{align*}
		\phi_\lambda^-:\R&\to\R^2,&
		\phi_\lambda^-(t)
		&:=
		\lambda
		\begin{pmatrix}
			-\frac{\sqrt{2}}{\cosh t}\\
			\lambda\tanh t
		\end{pmatrix}\fall\lambda\in\R;
	\end{align*}
	note that each $\phi_\lambda^-$ is homoclinic to $\phi_\lambda$. Hence, a branch of homoclinic solutions bifurcates from $\phi^\ast=0$ at $\lambda^\ast=0$ (see Fig.~\ref{figbranches1}). In order to confirm this scenario by means of \tref{thmsingle} we compute the partial derivative
	\begin{equation*}
		D_2f(t,x,\lambda)=\begin{pmatrix}
			-\tanh t & 0\\
			2x_1 & \tanh t
		\end{pmatrix}
	\end{equation*}
	leading to the variation equation $(V_{\lambda^\ast})$ explicitly given by
	\begin{equation*}
		\dot x
		= D_2f(t,\phi^\ast(t),\lambda^\ast)x=
		\begin{pmatrix}
			-\tanh t & 0\\
			0 & \tanh t
		\end{pmatrix}x
		\label{noed}
	\end{equation*}
	with the diagonal transition matrix
	\begin{align*}
		\Phi_{\lambda^\ast}(t,\tau)
		&=
		\begin{pmatrix}
			\tfrac{\cosh \tau}{\cosh t} & 0\\
			0 & \tfrac{\cosh t}{\cosh \tau}
		\end{pmatrix}
		\fall\tau,t\in\R.
	\end{align*}
	Therefore, the variation equation $(V_{\lambda^\ast})$ has exponential dichotomies on $\R_+$ with
	\begin{align*}
		P_{\lambda^\ast}^+(t)&\equiv
		\begin{pmatrix}
			1 & 0\\
			0 & 0
		\end{pmatrix},&
		R(P_{\lambda^\ast}^+(t))&\equiv\spann\set{e_1}
	\end{align*}
	and also on the semiaxis $\R_-$ with
	\begin{align*}
		P_{\lambda^\ast}^-(t)&\equiv
		\begin{pmatrix}
			0 & 0\\
			0 & 1
		\end{pmatrix},&
		N(P_{\lambda^\ast}^-(t))&\equiv\spann\set{e_1}.
	\end{align*}
	Whence, $(H_2)$ holds with $m^+=m^-=1$. Having established the Hypotheses $(H_0$--$H_2)$ we can compute an Evans function $E: \R\to\R$ defined on the entire real axis. Indeed, the variation equation \eqref{var} along $\phi_\lambda$ has the solutions
	\begin{equation*}
		\Phi_\lambda(t,0)\xi
		=
		\begin{pmatrix}
			\xi_1/\cosh t\\
			\sqrt{2}\lambda\tanh t\xi_1+\cosh t(2\sqrt{2}\lambda\arctan\tanh\tfrac{t}{2}\xi_1+\xi_2)
		\end{pmatrix}
	\end{equation*}
	and consequently \eqref{R-N} induces the dynamical characterizations:
	\begin{align*}
		\set{\xi\in\R^2:\,\sup_{0\leq t}\abs{\Phi_\lambda(t,0)\xi}<\infty}
		&=
		\set{\xi\in\R^2:\,\xi_2=-\sqrt{2}\lambda\tfrac{\pi}{2}\xi_1},\\
		\set{\xi\in\R^2:\,\sup_{t\leq 0}\abs{\Phi_\lambda(t,0)\xi}<\infty}
		&=
		\set{\xi\in\R^2:\,\xi_2=\sqrt{2}\lambda\tfrac{\pi}{2}\xi_1}.
	\end{align*}
	In conclusion, with \pref{propetagamma} it follows that
	\begin{equation*}
		E(\lambda)
		=
		\det\begin{pmatrix}
			1 & 1\\
			-\sqrt{2}\lambda\tfrac{\pi}{2} & \sqrt{2}\lambda\tfrac{\pi}{2}
		\end{pmatrix}
		=
		\sqrt{2}\pi\lambda
		\fall\lambda\in\R
	\end{equation*}
	is an Evans function. First, due to $E^{-1}(0)=\set{0}$ the splitting of the critical spectral interval guaranteed by \cref{corEvans} is illustrated (even upper semicontinuously) as
	\begin{equation*}
		\Sigma(\lambda)
		=
		\begin{cases}
			[-1,1],&\lambda=0,\\
			\set{-1}\cup\set{1},&\lambda\neq 0;
		\end{cases}
	\end{equation*}
	the critical spectral interval $[-1,1]$ of algebraic multiplicity $2$ splits into the singletons $\set{-1},\set{1}$ having algebraic multiplicity $1$. Second, because $E$ changes sign at $\lambda^\ast=0$, by \tref{thmsingle} there is a bifurcation of bounded entire solutions $\phi_\lambda^-$ to \eqref{cde} being homoclinic to $\phi_\lambda$. As demonstrated explicitly above, both branches $\phi_\lambda$ and $\phi_\lambda^-$ exist for all parameters; one has $\fB_{\R}=\set{0}$.
\end{example}
\section{Outlook and comparison}
\label{sec5}
The basic Fredholm theory from Sect.~\ref{sec2}, as well as the proof of \tref{thmmain} extends to further paths $T:[a,b]\to F_0(X,Y)$ of differential operators
\begin{equation}
	[T(\lambda)y](t):=\dot y(t)-A(t,\lambda)y(t)\text{ with coefficients }
	A(t,\lambda)\in\R^{d\tm d}
	\label{nopath}
\end{equation}
between appropriate pairs $(X,Y)$ of function spaces beyond those suitable for Carath{\'e}o\-dory equations \eqref{cde}. This adds the parity (and its applications) to the toolbox available for further areas, as discussed in e.g.\ \cite{alexander:gardner:jones:90, kapitula:promislow:13, sandstede:02} addressing PDEs or \cite{secchi:stuart:03} tackling nonautonomous Hamiltonian systems.

\paragraph{Ordinary differential equations}
Our central results literally carry over to nonautonomous ordinary differential equations \eqref{cde}, provided Hypothesis $(H_0)$ holds with continuous functions $f$ and $D_2f$ in their full set of variables. Then the functional analytical machinery presented here applies with $W^{1,\infty}(\R)$ and $L^\infty(\R)$ replaced by the corresponding subspaces of bounded continuous resp.\ continuously differentiable functions $BC^1(\R)$ and $BC(\R)$ resp.\ their subspaces of functions vanishing at $\pm\infty$; this is met in \eref{ex10}.

\paragraph{Difference equations}
Similarly, linearizing difference equations $x_{t+1}=f_t(x_t,\lambda)$ near branches of bounded entire solutions gives rise to operators
\begin{equation*}
	[T(\lambda)y]_t:=y_{t+1}-A_t(\lambda)y_t\text{ with }A_t(\lambda)\in\R^{d\tm d}.
\end{equation*}
Under corresponding dichotomy assumptions (cf.~\cite[pp.~101ff, Sect.~6.2]{anagnostopoulou:poetzsche:rasmussen:22}), $T(\lambda)$ can be shown to be an index $0$ Fredholm endomorphism on the spaces $\ell^\infty(\Z)$ of bounded sequences and the limit zero sequences $\ell_0(\Z)$. This allows to introduce an Evens function in this framework with the corresponding ramifications.

\paragraph{Parity and multiplicity}
The parity unifies several approaches extending the algebraic multiplicity $\bar\mu$ of critical parameters $\lambda_0$ from compact operators to paths of index $0$ Fredholm operators with invertible endpoints in terms of the relation
\begin{equation*}
	\sigma(T,[\lambda_0-\eps,\lambda_0+\eps])=(-1)^{\bar\mu}
	\quad\text{for sufficiently small }\eps>0.
\end{equation*}
Among them are the crossing number \cite[pp.~203ff]{kielhoefer:12} or the multiplicities due to Iz{\'e} \cite{ize:76}, Magnus \cite{magnus:76} or finally Esquinas \& L{\'o}pez-G{\'omez} \cite{esquinas:lopez:88}; their relation was studied in \cite[Thm.~1.4]{esquinas:88} and \cite{FiPejsachowiczV}. Indeed, the parity is invariant under Lyapunov--Schmidt reduction (cf.\ \cite{FiPejsachowiczV}). In our situation of paths having the form \eqref{nopath}, \tref{thmmain} connects the Evans function with these multiplicities via the relation $\sigma(T,\lambda_0)=(-1)^{\bar\mu}$ resp.\
\begin{equation*}
	(-1)^{\bar\mu}
	=
	\sgn E(\lambda_0-\eps)\sgn E(\lambda_0+\eps)
	\quad\text{for sufficiently small }\eps>0.
\end{equation*}
In addition, the product representation \cite[Prop.~5.6]{benevieri:furi:00} of the parity in terms of the sign of oriented Fredholm operators implies
\begin{equation*}
	\sgn T(\lambda_0-\eps)\sgn T(\lambda_0+\eps)
	=
	\sgn E(\lambda_0-\eps)\sgn E(\lambda_0+\eps)
	\text{ for sufficiently small }\eps>0.
\end{equation*}

We finally point out a useful extension of results stemming from the abstract set-up of App.~\ref{appB}. Indeed, \cite[Thm.~6.18]{FiPejsachowiczV} establish that if a path $T:[a,b]\to F_0(X,Y)$ is differentiable in $\lambda_0\in(a,b)$ and satisfies the splitting
\begin{equation}\label{split}
	\dot T(\lambda_0)N(T(\lambda_0))\oplus R(T(\lambda_0))=Y,
\end{equation}
then $\lambda_0$ is an isolated singular point of $T$ and for sufficiently small $\varepsilon>0$ one has
\begin{equation}\label{dim-ker}
	\sigma(T,[\lambda_0-\varepsilon,\lambda_0+\varepsilon])=(-1)^{\dim N(T(\lambda_0))}.
\end{equation}
Observe that under \eqref{split} and \eqref{dim-ker} one has the equivalence
\begin{equation*}
	\sigma(T,[\lambda_0-\varepsilon,\lambda_0+\varepsilon])=-1
	\quad\Leftrightarrow\quad
	\dim N(T(\lambda_0)) \text{ is odd.}
\end{equation*}
In contrast, based on our approach a path $T$ neither has to be differentiable nor must satisfy \eqref{split} in $\lambda_0$. Beyond that $\sigma(T,[\lambda_0-\varepsilon,\lambda_0+\varepsilon])=-1$ may hold for even dimensions of $N(T(\lambda_0))$. The above \eref{example-9} illustrates that such a situation can actually occur.
\begin{example}
	Let $X=W^{1,\infty}(\R)$ and $Y=L^{\infty}(\R)$. In the framework of \eref{example-9} the path $T:[a,b]\to L(W^{1,\infty}(\R),L^{\infty}(\R))$ discussed in \tref{thmmain} becomes explicitly
	\begin{equation*}
		[T(\lambda)y](t)
		=
		\dot{y}(t)-\begin{pmatrix}
			a(t)I_n & 0\\
			C(\lambda) & -a(t)I_n
		\end{pmatrix}y(t)\quad\text{for a.a.\ }t\in\R
	\end{equation*}
	and $y\in W^{1,\infty}(\R)$. Even if the coefficient function $C:[a,b]\to\R^{n\tm n}$ is merely assumed to be continuous, an Evans function for \eqref{var} can be constructed. Yet, unless $C$ is differentiable in some $\lambda_0\in(a,b)$, the condition \eqref{split} cannot be employed. Beyond that, even for $C$ being differentiable at $\lambda_0$, but $\dot C(\lambda_0)\not\in GL(\R^n,\R^n)$, then also the path $T$ is differentiable in $\lambda_0$ with $\dot T(\lambda_0)\in L(W^{1,\infty}(\R),L^{\infty}(\R))$ given by
	\begin{equation*}
		[\dot T(\lambda_0)y](t)
		=
		-\begin{pmatrix}
			0 & 0\\
			\dot C(\lambda_0) & 0
		\end{pmatrix}y(t) \quad\text{for a.a.\ }t\in\R.
	\end{equation*}
	But it is clear that this derivative does not fulfill a splitting \eqref{split}.
\end{example}
\appendix
\renewcommand{\theequation}{\Alph{section}.\arabic{equation}}
\renewcommand{\thesection}{\Alph{section}}
\section*{Appendices}
Assume that $X,Y$ are real Banach spaces. For the convenience of the reader we briefly review the construction of the parity for a path of Fredholm operators and provide its properties, as well as applications in bifurcation theory from \cite{Fitz91, FiPejsachowiczIII, FiPejsachowiczIV, FiPejsachowiczV,FitPejRab}.
\section{The parity}
\label{appA}
We denote a continuous function $T: [a,b]\to L(X,Y)$ as a \emph{path}. It is said to have \emph{invertible endpoints}, if moreover $T(a),T(b)\in GL(X,Y)$ holds. Referring to \cite{fitzPeja86}, for each path $T: [a,b]\to F_0(X,Y)$ there exists a path $P:[a,b]\to GL(Y,X)$ having the property that $P(\lambda)T(\lambda)-I_X\in L(X,X)$ is a compact operator for every $\lambda\in[a,b]$; such a function $P$ is called \emph{parametrix} for $T$. In case $T: [a,b]\to F_0(X,Y)$ has invertible endpoints, then its \emph{parity} on $[a,b]$ is defined as
\begin{equation*}
	\sigma(T,[a,b]):=\deg_{LS}(P(a)T(a))\cdot \deg_{LS}(P(b)T(b))\in\set{-1,1},
\end{equation*}
where the symbol $\deg_{LS}$ denotes the Leray--Schauder degree (cf.\ e.g.\ \cite[pp.~199ff]{kielhoefer:12}).

We understand paths $T,S:[a,b]\to F_0(X,Y)$ with invertible endpoints as \emph{homotopic}, if there exists a continuous map $h:[0,1]\tm [a,b]\to F_0(X,Y)$ with the properties
\begin{itemize}
	\item $h(0,\lambda)=T(\lambda)$ and $h(1,\lambda)=S(\lambda)$ for all $\lambda\in[a,b]$,

	\item $h(t,\cdot):[a,b]\to F_0(X,Y)$ has invertible endpoints for all $t\in(0,1)$.
\end{itemize}
\begin{lemma}[properties of the parity]\label{lemparity}
	Let $E,F$ and $Z$ be further real Banach spaces and assume $T: [a,b]\to F_0(X,Y)$ is a path with invertible endpoints.
	\begin{enumerate}
		\item[(a)] Homotopy invariance \emph{\cite[p.~54, (6.11)]{FiPejsachowiczIII}}: If $T$ is homotopic to a further path\\ $S:[a,b]\to F_0(X,Y)$ with invertible endpoints, then $\sigma(T,[a,b])=\sigma(S,[a,b])$.

		\item[(b)] Multiplicativity under partition of $[a,b]$ \emph{\cite[p.~53, (6.9)]{FiPejsachowiczIII}}: If $T(c)\in GL(X,Y)$ for some $c\in(a,b)$, then $\sigma(T,[a,b])=\sigma(T,[a,c])\cdot \sigma(T,[c,b])$.

		\item[(c)] Multiplicativity under composition \emph{\cite[p.~54, (6.10)]{FiPejsachowiczIII}}: If $S:[a,b]\to F_0(Y,Z)$ is a path with invertible endpoints, then\footnote{we abbreviate $(TS)(\lambda):=T(\lambda)S(\lambda)$ for all $\lambda\in[a,b]$} $\sigma(ST,[a,b])=\sigma(S,[a,b])\cdot\sigma(T,[a,b])$.

		\item[(d)] Multiplicativity under direct sum \emph{\cite[p.~54, (6.12)]{FiPejsachowiczIII}}: If $S:[a,b]\to F_0(E,F)$ is a path with invertible endpoints, then
		$
			\sigma(T\oplus S,[a,b])=\sigma(T,[a,b])\cdot \sigma(S,[a,b]).
		$

		\item[(e)] Finite-dimensional case \emph{\cite[p.~53, (6.8)]{FiPejsachowiczIII}}: If $X=Y$ and $\dim X<\infty$, then
		$
			\sigma(T,[a,b])=\sgn \det T(a)\cdot\sgn \det T(b).
		$

		\item[(f)] Triviality property \emph{\cite[p.~52, Thm.~6.4]{FiPejsachowiczIII}}: $\sigma(T,[a,b])=1$ if and only if the path $T:[a,b]\to F_0(X,Y)$ can be deformed in $F_0(X,Y)$ to a path in $GL(X,Y)$ through a homotopy with invertible endpoints. In particular, if $T(\lambda)\in GL(X,Y)$ for all $\lambda\in [a,b]$, then $\sigma(T,[a,b])=1$.
	\end{enumerate}
\end{lemma}

For actual parity computations the following result is crucial:
\begin{lemma}[reduction property of the parity, \cite{FiPejsachowiczIII, FitPejRab}]\label{lemparity2}
	Let $T:[a,b]\to F_0(X,Y)$ be a path with invertible endpoints. If $V$ is a finite-dimensional subspace of $Y$ which satisfies $T(\lambda)X+V=Y$ for all $\lambda\in [a,b]$, then
	$
		E(T,V):=\set{(\lambda,x)\in [a,b]\tm X\mid T(\lambda)x\in V}
	$
	has the following properties:
	\begin{enumerate}
		\item[(a)] $E(T,V)$ is a subbundle of $[a,b]\tm X$ with fibers $E(T,V)_{\lambda}=T(\lambda)^{-1}V$. In particular, $\dim T(\lambda)^{-1}V=\dim V$ for all $\lambda\in [a,b]$.

		\item[(b)] For every bundle trivialization $\hat T:[a,b]\tm V\to E(T,V)$ one has
		\begin{align*}
			T\circ\hat T_{a},T\circ\hat T_{b}\in GL(V,V)\text{ and }\sigma(T,[a,b])=\sigma(T\circ\hat T,[a,b]),
		\end{align*}
		where $T\circ\hat T:[a,b]\to L(V,V)$ is given by $(T\circ\hat T)_{\lambda}(v):=T(\lambda)\hat T(\lambda,v)$.
	\end{enumerate}
\end{lemma}
Since the interval $[a,b]$ is contractible, every such vector bundle $E(T,V)$ over $[a,b]$ possesses a bundle trivialization $\hat T: [a,b]\tm V\to E(T,V)$ (cf.\ \cite[p.~30, Cor.~4.8]{husemoller:74}).

\begin{remark}
	For the sake of a detailed description of the path $T\circ\hat T: [a,b]\to L(V,V)$, let $v_1,\ldots,v_d$ be a basis of the subspace $V$ from \lref{lemparity2}. Since $E(T,V)$ is a vector bundle over $[a,b]$ with fibers isomorphic to $V$, it follows that there exist continuous sections $\varphi_1,\ldots,\varphi_d:[a,b]\to E(T,V)$ such that $\varphi_1(\lambda),\ldots,\varphi_d(\lambda)$ forms a basis of $E(T,V)_{\lambda}$ for all $\lambda\in [a,b]$. We define an isomorphism
	\begin{align*}
		\hat T: [a,b]\times V&\to E(T,V),&
		\hat T_\lambda(v)
		&=
		\hat T_{\lambda}\left(\sum_{i=1}^d \alpha_i v_i\right)
		:=
		\sum_{i=1}^d\alpha_i \varphi_i(\lambda)
	\end{align*}
	and consider functionals $v_1^{\ast},\ldots,v_d^{\ast}:V\to\R$ uniquely determined from the conditions $v_j^{\ast}(v_i)=\delta_{ij}$, $1\leq i,j\leq d$. Then
$(T\circ\hat T)_{\lambda}: V\to V$ can be represented as matrix
	\begin{align*}
		M(\lambda)&=(m_{ij}(\lambda))_{i,j=1}^d,&
		m_{ij}(\lambda)&:=\sprod{v_j^{\ast},T(\lambda)(\varphi_i(\lambda))}
		\fall 1\leq i,j\leq d
	\end{align*}
	and \lref{lemparity}(e) and \ref{lemparity2} imply $\sigma(T,[a,b])=\sgn\det M(a)\cdot\sgn\det M(b)$.
\end{remark}

Our subsequent bifurcation result requires a local version of the parity near \emph{isolated singular points} $\lambda_0\in(a,b)$. This means $T(\lambda_0)\not\in GL(X,Y)$, but there exists a neighborhood $\Lambda_0\subseteq(a,b)$ of $\lambda_0$ such that $T(\lambda)\in GL(X,Y)$ for all $\lambda\in\Lambda_0\setminus\set{\lambda_0}$, which allows us to introduce the \emph{parity index}
\begin{equation*}
	\sigma(T,\lambda_0):=\lim_{\eps\searrow 0}\sigma(T,[\lambda_0-\eps,\lambda_0+\eps]).
\end{equation*}
\section{Parity and local bifurcations}
\label{appB}
Assume that $U\subseteq X\tm\R$ is nonempty and open. We investigate abstract parametrized equations
\begin{equation}
	\tag{$O_\lambda$}
	G(x,\lambda)=0
\end{equation}
for continuous functions $G:U\to Y$ having the following properties:
\begin{enumerate}
	\item[$\mathbf{(M_1)}$] the partial derivative $D_1G:U\to L(X,Y)$ exists as continuous function,

	\item[$\mathbf{(M_2)}$] there exists an open interval $\Lambda\subseteq\R$ with $\set{0}\tm\Lambda\subseteq U$ such that $G(0,\lambda)=0$ and $D_1G(0,\lambda)\in F_0(X,Y)$ for all $\lambda\in\Lambda$.
\end{enumerate}
We denote a $\lambda_0\in\Lambda$ as \emph{bifurcation value}, provided $(0,\lambda_0)$ is a bifurcation point for \eqref{abs} (e.g.\ \cite[p.~309, Def.~1]{zeidler:95}).
\begin{theorem}[local bifurcations]\label{thmbifurcation}
	Let $(M_1$--$M_2)$ hold. If $\lambda_0\in\Lambda$ is an isolated singular point of $D_1G(0,\cdot)$ with parity index
	\begin{equation*}
		\sigma(D_1G(0,\cdot),\lambda_0)=-1,
	\end{equation*}
	then $\lambda_0$ is a bifurcation value for \eqref{abs}. More precisely, there exists a $\delta_0>0$ such that for each $\delta\in(0,\delta_0)$ a connected component
	\begin{equation*}
		\cC
		\subseteq
		G^{-1}(0)\setminus\set{(0,\lambda)\in X\tm\Lambda:\,D_1G(0,\lambda)\in GL(X,Y)}
	\end{equation*}
	joins the set $\set{(0,\lambda)\in X\tm(\lambda_-,\lambda_+):\,D_1G(0,\lambda)\not\in GL(X,Y)}$ of critical trivial solutions to the surface $\bigl\{(x,\lambda)\in X\tm\Lambda:\,\norm{x}_X=\delta\bigr\}$.
\end{theorem}
\begin{proof}
	Since $U\subseteq X\tm\Lambda$ is open, there exist open neighborhoods $U_0\subseteq X$ of $0$ and $\Lambda_0\subseteq\Lambda$ of $\lambda_0$, so that $U_0\tm\Lambda_0\subseteq U$. Because $\lambda_0$ is assumed to be an isolated singular point of $D_1G(0,\cdot)$, there exist parameters $\lambda_-<\lambda_+$ in $\Lambda_0$ yielding invertible endpoints $D_1G(0,\lambda_-),D_1G(0,\lambda_+)\in GL(Y,X)$ and parity $\sigma(D_1G(0,\cdot),[\lambda_-,\lambda_+])=-1$. Now for $C^1$-mappings $G:U\to Y$ this allows an immediate application of \cite[Thm.~4.1]{lopez:sampedro:23} under the assumption that the generalized algebraic multiplicity of the path $D_1G(\cdot,0)$ is odd. But because of \cite[Thm.~3.2]{lopez:sampedro:23} this is equivalent to our assumption of having a parity index $-1$ in $\lambda_0$. Moreover, the continuous differentiability of $G$ can be weakened to our assumption $(M_2)$ using methods due to Pejsachowicz \cite[Lemma~2.3.1]{pejsachowicz:11a} or \cite[Lemma~6.3]{Pej-Ski} together with \cite[Thm.~8.73]{vaeth:12}.
\end{proof}
\bibliographystyle{amsplain}
\providecommand{\bysame}{\leavevmode\hbox to3em{\hrulefill}\thinspace}

\end{document}